\documentclass[12pt,notitlepage]{amsart}
\usepackage{latexsym,amsfonts,amssymb,amsmath,amsthm}
\usepackage{color}
\usepackage[inner=1.0in,outer=1.0in,bottom=1.0in, top=1.0in]{geometry}

\newcommand{\mz}{\ensuremath{\mathbb Z}}
\newcommand{\mr}{\ensuremath{\mathbb R}}

\newcommand{\shortmod}{\ensuremath{\negthickspace \negthickspace \negthickspace \pmod}}

\newcommand{\half}{\ensuremath{ \frac{1}{2}}}

\newcommand{\intR}{\int_{-\infty}^{\infty}}

\newcommand{\thalf}{\tfrac12}
\newcommand{\sumstar}{\sideset{}{^*}\sum}
\newcommand{\leg}[2]{\left(\frac{#1}{#2}\right)}
\newcommand{\e}[2]{e\left(\frac{#1}{#2}\right)}

\DeclareMathOperator{\sgn}{sgn}

\theoremstyle{plain}		
	\newtheorem{mytheo}{Theorem}[section]
	
	\newtheorem{myprop}[mytheo]{Proposition}
	
     \newtheorem{mylemma}[mytheo]{Lemma}
	
	\newtheorem{myconj}[mytheo]{Conjecture}
	\newtheorem{myremark}[mytheo]{Remark}

\theoremstyle{remark}

\numberwithin{equation}{section}
\begin{document}
\title{The third moment of quadratic Dirichlet L-functions}
\author{Matthew P. Young} 
\address{Department of Mathematics \\
	  Texas A\&M University \\
	  College Station \\
	  TX 77843-3368 \\
		U.S.A.}
\email{myoung@math.tamu.edu}
\thanks{This material is based upon work supported by the National Science Foundation under agreement No. DMS-0758235.  Any opinions, findings and conclusions or recommendations expressed in this material are those of the authors and do not necessarily reflect the views of the National Science Foundation.}

\begin{abstract}
We study the third moment of quadratic Dirichlet $L$-functions, obtaining an error term of size $O(X^{3/4 + \varepsilon})$.
\end{abstract}

\maketitle
\section{Introduction}
Quadratic twists of $L$-functions have been extensively studied from many points of view.  
In this paper, we consider the problem of the third moment of quadratic Dirichlet $L$-functions.  Soundararajan \cite{Sound} was the first to find an asymptotic formula for this third moment, obtaining a power savings in the error term.  Subsequently, Diaconu-Goldfeld-Hoffstein \cite{DGH} used the multiple Dirichlet series method and obtained a superior error term.  Our goal in this paper is to further refine the error term in this problem by adapting a technique introduced in \cite{Y}.
\begin{mytheo}
\label{thm:mainthm}
Let $F$ be a smooth, compactly-supported function on the positive reals with support in a dyadic interval $[X/2, 3X]$, and satisfying
\begin{equation}
\label{eq:Fproperty}
 F^{(j)}(x) \ll_j X^{-j}, \qquad j=0, 1, 2, \dots.
\end{equation}
Then
\begin{equation}
\label{eq:thirdmoment}
 \sumstar_{(d,2)=1} L(1/2, \chi_{8d})^3 F(d) = \sumstar_{(d,2)=1} P(\log{d}) F(d) + O(X^{3/4 + \varepsilon}),
\end{equation}
where the star indicates the sum is over squarefree integers, and $P(x)$ is a certain degree $6$ polynomial.
\end{mytheo}
For the smoothed sum as above, Soundararajan \cite{Sound} previously obtained an error of size $X^{7/8 + \varepsilon}$ (for this, see the final displayed equation on p.487 and optimally choose $Y = X^{1/8}$).  Using the multiple Dirichlet series method, \cite{DGH} obtained an error of size $X^{4/5 + \varepsilon}$ for the smoothed mean value.

We shall present the full details of the proof of Theorem \ref{thm:mainthm}, but the technique can certainly be generalized to cover other cases.  For instance, one can find an asymptotic formula for the second moment of the same family, perhaps with an error term of size $O(X^{1/2 + \varepsilon})$.  It would be interesting to study the first moment of quadratic twists of a $GL_3$ automorphic form such as the symmetric-square lift of a holomorphic modular form on the full modular group (we mention this case since the Ramanujan conjecture is known for such forms which avoids some potential pitfalls).  This case should be largely similar to the third moment in this paper, but there is an extra difficulty since one cannot exploit the factorization of the $L$-functions and use Heath-Brown's quadratic large sieve \cite{H-B} (see comments following Theorem \ref{thm:recursive} below).

We stress that the condition that $d$ is squarefree is a substantial difficulty (improperly seen as a simple technicality).  For instance, with $F$ as in Theorem \ref{thm:mainthm}, compare the difficulty in the following two sums:
\begin{equation}
A= \sum_{d \in \mz} F(d), \qquad B= \sumstar_{d \in \mz} F(d).
\end{equation}
Poisson summation quickly shows $A = \widehat{F}(0) + O(X^{-C})$ for any large $C > 0$.  On the other hand, the standard Perron-type formula approach quickly shows $B = \frac{\widehat{F}(0)}{\zeta(2)} + O(X^{1/2})$, and any improvement on the exponent $1/2$ would give a quasi-Riemann hypothesis, that is, a zero-free region for the Riemann zeta function in $\text{Re}(s) > 1-\delta$ for some $\delta >0$.  The underlying reason for this is that the generating function for the squarefree numbers is $\zeta(s)/\zeta(2s)$.  

The method of proof of Theorem \ref{thm:mainthm} uses the main idea of \cite{Y} which is a kind of recursive argument that efficiently treats the squarefree condition.  In effect, there is almost no cost to handling squarefree numbers compared to all integers (up to a barrier at improving on an error of size $O(X^{1/2 + \varepsilon})$; our method unfortunately does not lead to a quasi-Riemann hypothesis!).

Here we briefly sketch the method.  We start with an approximate functional equation for the central value (actually we consider values slightly shifted from the central point).
The basic idea is to use M\"{o}bius inversion on the sum over $d$ to remove the squarefree condition, and then to use Poisson summation.  The M\"{o}bius procedure effectively shortens the sum over $d$: say we write $\sumstar_d f(d) = \sum_{a} \mu(a) \sum_{d} f(a^2 d)$, then for large values of $a$ the sum over $d$ would be quite short indeed, and then Poisson summation might not be a wise choice as the dual sum would be longer than the original sum.  The idea is to treat small and large $a$ differently.  For $a$ small one uses Poisson summation; for $a$ large one could use the trivial bound as in \cite{Sound}.  However,
for $a$ large one can obtain a more precise result by writing $d \rightarrow b^2 d$ where the new $d$ is squarefree.  Then one arrives at something similar to where one started, and rather than bounding this expression trivially one can plug in an already-obtained asymptotic formula.  Doing so leads to certain partial main terms that turn out to combine with other main terms from $a$ small.  This was the main new idea appearing in \cite{Y}.  The idea is simple yet carrying out the details involves proving elaborate (perhaps miraculous?) combinatorial type identities for the various main terms that arise in different ways.  

We attempted to perform the same technique on the third moment in this paper, and the method largely succeeds though there are some new difficulties.  One difficulty is that the main terms do not combine quite as nicely as in \cite{Y};  apparently we are missing some harmonics.  This seems to be related to the problem of getting an asymptotic formula for the fourth moment of this family (where if one uses Poisson summation and collects all the main terms arising in similar ways as in this paper, then one does not obtain all the main terms predicted by \cite{CFKRS}).  
To get around this issue of missing terms (which are not main terms but rather unwieldly expressions that should combine and simplify), we noticed that a particular choice of a weight function in the approximate functional equation can cause some of these terms to vanish.

Obtaining an error term better than $O(X^{3/4})$ is of great interest for a few reasons.  For one, it is apparently related to the subconvexity problem for a $GL_3$ $L$-function twisted by a quadratic character.  The basic idea is that if the central values were non-negative, and the weight function in Theorem \ref{thm:mainthm} were allowed to be supported in a short interval (as in \cite{Sound}), then we could drop all but one term in the moment and obtain an upper bound of size $O(X^{3/4-\delta})$ for that $L$-value, which would be a subconvexity bound as the conductor is $\asymp X^3$ (note that recently Blomer \cite{Blomer} obtained subconvexity for quadratic twists of a symmetric-square lift of a Maass or holomorphic form of full level, though not using the above approach).  Although we cannot rigorously claim that improving the error term in this way implies subconvexity, it seems that the issues are related.

The combinatorial identities in this paper are quite elaborate.  This is a general phenomenon observed in various examples yet not well understood.  In some previous works we found some techniques that simplify the proofs of these identities, which involve showing that two very different arithmetical Dirichlet series turn out to be the same, at least for special values of the parameters.  Because of the sizes of the expressions involved here, we resorted to computer verification of certain identities.  We stress that the remarkable simplifications in this paper give very strong evidence to their correctness; our experience shows that even a tiny error completely destroys the simplifications.

As in any proof by induction, it is extremely beneficial to have the inductive hypothesis provided ahead of time.  This is accomplished by the general moment conjectures of \cite{CFKRS}; their method leads to the form of the main term very quickly and with minimal effort.

It is interesting to compare how our approach is related to the multiple Dirichlet series method used in \cite{DGH}.  To this end, we briefly summarize their method in simplified terms.  One might wish to study the two-variable Dirichlet series
\begin{equation*}
Z(s,w) = \sumstar_{d \geq 1, \text{ odd}} \frac{L(s, \chi_{8d})^3}{d^w}.
\end{equation*}
By the Perron formula method, the meromorphic continuation of this function in terms of $w$, with $s =1/2$, is intimately related to the error term in \eqref{eq:thirdmoment}.  In fact, Theorem \ref{thm:mainthm} gives the meromorphic continuation of $Z(1/2, w)$ to $\text{Re}(w) > 3/4$ with a singularity at $w=1$ only (we include a brief sketch of this fact in the next paragraph). One approach to studying $Z(s,w)$, initially in its region of absolute convergence, would proceed by writing the Dirichlet series expansion for $L(s, \chi_{8d})^3$, using M\"{o}bius inversion to remove the condition that $d$ is squarefree, and then applying Poisson summation to the sum over $d$.  Heuristically, this should lead to a connection between $Z(s,w)$ and $Z(s,1-w)$, but unfortunately the Poisson formula does not lead to exactly the same Dirichlet series (we see this in our work in Section \ref{section:MNcalculation} and Lemma \ref{lemma:Jprop} below).  Instead, the method of \cite{DGH} is to construct a different auxiliary Dirichlet series say $Z'(s,w)$ that does satisfy nice functional equations but differs from $Z(s,w)$ in that it involves non-fundamental discriminants and corresponding correction factors at ``bad'' primes in the Euler products.  The functional equations of $Z'$ make it a pleasant function to understand, and with a kind of sieving argument this information can be transferred to $Z$ (and hence to the moment on the left hand side of \eqref{eq:thirdmoment}).

To see that $Z(1/2, w)$ has the desired meromorphic continuation, let $F$ be a function satisfying the conditions in Theorem \ref{thm:mainthm} with $X=1$, and such that $F(x) \geq 0$ for all $x$ (but not identically zero).  With $\widetilde{F}$ denoting the Mellin transform of $F$, we have $\widetilde{F}(w) d^{-w} = \int_0^{\infty} x^{-w} F(d/x) \frac{dx}{x}$.  Note that the real part of $\widetilde{F}(w)$ is always positive (this is direct from the definition) and hence this function has no zeros.
Then
\begin{equation*}
\widetilde{F}(w) Z(1/2, w) = \int_0^{\infty} x^{-w} \Big( \sumstar_{d \geq 1, \text{ odd}} L(1/2, \chi_{8d})^3 F(d/x) \Big) \frac{dx}{x}.
\end{equation*}
Applying \eqref{eq:thirdmoment}, we obtain a main term plus an error term.   The main term can be explicitly computed which gives the meromorphic continuation of this term, while the error term has absolute convergence for $\text{Re}(w) > 3/4$ and hence leads to the desired analytic continuation.

It would be desirable to extend the class of allowable weight functions in Theorem \ref{thm:mainthm} in order to obtain non-vanishing results in short intervals, or to un-smooth the sum in \eqref{eq:thirdmoment}.  To obtain strong results in this direction, it would be desirable to study moments analogous to \eqref{eq:thirdmoment} but at points $s$ with 
 potentially large imaginary part.  The multiple Dirichlet series approach is well-suited to keeping track of this $s$-dependence.
As an aside, we mention that un-smoothing \eqref{eq:thirdmoment} is somewhat subtle because the central values are not known to be non-negative; as a result, the error term as stated in Theorem $2$ of \cite{Sound} is not justified without this assumption.  It would also be valuable to understand the other classes of fundamental discriminants in place of those of the form $8d$ with $d$ odd squarefree and positive; in principle, the different cases should be similar to each other.


\subsection{Acknowledgements}
I thank Brian Conrey, David Farmer, Jeff Hoffstein, Mike Rubinstein, and K. Soundararajan for discussions.

\section{Tools}
In this section we quote some results we use throughout the paper.
For $d$ a fundamental discriminant, let $\chi_d$ be the corresponding primitive quadratic Dirichlet character.  We shall work with discriminants of the form $8d$ where $d$ is odd, squarefree, and positive, so that $\chi_{8d}(n) = \leg{8d}{n}$ for $n$ odd is an even, primitive character of conductor $8d$.

\subsection{Functional equation}
The functional equation of the Dirichlet $L$-function associated to $\chi_{8d}$ is
\begin{equation}
\Lambda(s, \chi_{8d}) = \left(\frac{8d}{\pi}\right)^{s/2} \Gamma\left(\frac{s}{2} \right) L(s, \chi_{8d})
 = \Lambda(1-s, \chi_{8d}).
\end{equation}
In its asymmetric form it reads
\begin{equation}
L(s, \chi_{8d}) = X(s) L(1-s, \chi_{8d}),
\qquad 
X(1/2 + u) = \left(\frac{8d}{\pi} \right)^{-u } \frac{\Gamma\left(\frac{1/2 - u}{2} \right)}{\Gamma\left(\frac{1/2 + u }{2} \right)}.
\end{equation}

\subsection{Approximate functional equation}
\label{section:AFE}
Suppose that $\alpha, \beta, \gamma$, called the ``shift parameters'', are small complex numbers.  We shall suppose initially that each $\lambda \in \{\alpha, \beta, \gamma \}$ lies in a punctured rectangle of the form $|\text{Re}(\lambda)| \leq c_1/\log{X}$, $|\text{Im}(\lambda)| \leq c_2 X^{\varepsilon}$ minus $|\text{Re}(\lambda)| \leq c_1/(2\log{X})$, $|\text{Im}(\lambda)| \leq (c_2/2) X^{\varepsilon}$
where the $c_i$ (depending on the choice of $\lambda$) are chosen so that the domain corresponding to each shift parameter has distance $\gg 1/X^{\varepsilon}$ from the other two shift parameters.  Later we can relax this condition.

\begin{myprop}[Approximate functional equation]
\label{prop:AFE}
Let $G(s)$ be an entire, even function with rapid decay in the strip $|\text{Re}(s)| \leq 10$.  Then for $\chi_{8d}$ as above,
\begin{multline}
L(\thalf + \alpha, \chi_{8d}) L(\thalf + \beta, \chi_{8d}) L(\thalf + \gamma, \chi_{8d})
=
\sum_{n} \frac{\chi_{8d}(n) \sigma_{\alpha,\beta,\gamma}(n)}{n^{\thalf}} V_{\alpha, \beta,\gamma} \left(\frac{n}{d^{3/2}}\right)
\\
+ d^{-\alpha-\beta-\gamma} \Gamma_{\alpha,\beta,\gamma} \sum_n \frac{\chi_{8d}(n) \sigma_{-\alpha,-\beta,-\gamma}(n)}{n^{\thalf}} V_{-\alpha, -\beta,-\gamma} \left(\frac{n}{d^{3/2}}\right),
\end{multline}
where notation is as follows:
\begin{equation}
\label{eq:Vdef}
V_{\alpha, \beta,\gamma}(x) = \frac{1}{2 \pi i} \int_{(1)} \frac{G(s)}{s} g_{\alpha, \beta, \gamma}(s) x^{-s} ds;
\end{equation}
\begin{equation}
g_{\alpha, \beta,\gamma}(s) = \left(\frac{8}{\pi}\right)^{\frac{3s}{2}}
\frac{\Gamma\left(\frac{\half + \alpha + s}{2} \right)}{\Gamma\left(\frac{\half + \alpha}{2} \right)} 
\frac{\Gamma\left(\frac{\half + \beta + s}{2} \right)}{\Gamma\left(\frac{\half + \beta}{2} \right)} 
\frac{\Gamma\left(\frac{\half + \gamma + s}{2} \right)}{\Gamma\left(\frac{\half + \gamma}{2} \right)};
\end{equation}
$\Gamma_{\alpha,\beta,\gamma}= \Gamma_{\alpha} \Gamma_{\beta} \Gamma_{\gamma} $, where
\begin{equation}
\label{eq:Gammadef}
 \Gamma_{\alpha} = \left(\frac{8}{\pi} \right)^{-\alpha} 
\frac{\Gamma\left(\frac{\thalf- \alpha}{2} \right)}{\Gamma\left(\frac{\thalf + \alpha}{2} \right)};
\end{equation}
\begin{equation}
 \sigma_{\alpha,\beta,\gamma}(n) = \sum_{a b c = n} a^{-\alpha} b^{-\beta} c^{-\gamma}.
\end{equation}
\end{myprop}
\begin{myremark}
\label{remark:zero}
 We shall choose $G$ to vanish at the poles of $\zeta(1 + 2\alpha + 2s)$, $\zeta(1 + 2 \beta + 2s)$, etc., and to be divisible by
a large collection of values of $\xi(s) = s(1-s) \pi^{-s/2} \Gamma(s/2) \zeta(s)$, including $\xi(4 + 2\alpha + 2\beta + 4s)$, $\xi(2 + 2\alpha + 2 \gamma + 4s)$, etc.  (here we have listed the first couple terms in the numerator and denominator, respectively, of \eqref{eq:Ap'product} with $\alpha, \beta, \gamma$, replaced by $\alpha + s, \beta + s, \gamma+s$).
More precisely, these zeta factors occur in the analysis of the arithmetical
factor $A_{\alpha + s, \beta + s, \gamma+s}$ 
defined by \eqref{eq:Aoriginal} and developed much further in Lemma \ref{lemma:A} below.  For a given $\varepsilon > 0$ small we choose $G(s)$ 
to vanish at the poles of all the $\zeta$'s which occur in Lemma \ref{lemma:A} as numerators (i.e., with $d_{a,b,c} > 0$), and also
to be divisible by all the $\zeta$'s which occur in Lemma \ref{lemma:A} as denominators (i.e., with $d_{a,b,c} < 0$), when obtaining the meromorphic continuation of $A_{\alpha + s, \beta+s, \gamma+s}$ to $\text{Re}(s) > -\half + \varepsilon$.  

We furthermore suppose by a symmetrization argument that $G(s)$ is symmetric under any permutation of $\{\alpha, \beta, \gamma\}$, and under switching any $\alpha, \beta, \gamma$ with its negative, and under switching $s$ with $-s$.  Then by scaling we can ensure $G(0) = 1$; Assuming the shift parameters lie in the punctured rectangles as above then in terms of the shift parameters we have $G(s) \ll X^{\varepsilon}$.

\end{myremark}
Making such a choice for $G$ remarkably simplifies certain later computations.  The point is that we wish to use a contour shift argument to analyze certain integrals in the development of the moment under consideration.  In so doing one would cross many poles arising from the arithmetical factors.  However, the contribution from these residues would involve values of $G(s)$ at points other than $s=0$ and since $G$ can be chosen from a wide class of functions, it is apparently unlikely that these residues can persist as main terms.  By having $G$ vanish at these points we see a priori that these terms do not survive in the final answer.  In \cite{Y}, we did not move the contours past such poles but instead matched up integrals to form the simplifications.  One might hope to see a similar matching here but in fact it seems some of the terms are missing, that is, the method does not capture all these terms.  One would naturally speculate that these missing terms arise from an incomplete analysis of the new mean value arising after Poisson summation (i.e., in \eqref{eq:M1evaluated} below).

\subsection{Poisson summation}
We now quote Soundararajan's result
\begin{mylemma}[\cite{Sound}]
\label{lemma:poisson}
 Let $F$ be a smooth function with compact support on $\mr^{+}$, and suppose that $n$ is an odd integer.  Then
\begin{equation}
 \sum_{(d,2) = 1} \leg{d}{n} F(d) = \frac{1}{2n} \leg{2}{n} \sum_{k \in \mz} (-1)^k G_k(n) \check{F}\left(\frac{k}{2n}\right),
\end{equation}
where
\begin{equation}
 G_k(n) = \left(\frac{1-i}{2} + \leg{-1}{n} \frac{1+i}{2}\right) \sum_{a \shortmod{n}} \leg{a}{n} \e{ak}{n},
\end{equation}
and
\begin{equation}
 \check{F}(y) = \intR (\cos(2 \pi x y) + \sin(2 \pi x y)) F(x) dx.
\end{equation}
\end{mylemma}
The Gauss-type sum is calculated exactly with the following
\begin{mylemma}[\cite{Sound}]
\label{lemma:Gk}
If $m$ and $n$ are relatively prime odd integers, then $G_k(mn) = G_k(m) G_k(n)$, and if $p^{\alpha}$ is the largest power of $p$ dividing $k$ (setting $\alpha=\infty$ if $k=0$), then
\begin{equation}
\label{eq:Gk}
 G_k(p^{\beta}) = 
\begin{cases}
 0, \qquad & \text{if $\beta \leq \alpha$ is odd}, \\
 \phi(p^{\beta}), \qquad & \text{if $\beta \leq \alpha$ is even}, \\
 -p^{\alpha}, \qquad & \text{if $\beta = \alpha + 1$ is even}, \\
 \leg{kp^{-\alpha}}{p} p^{\alpha} \sqrt{p}, \qquad & \text{if $\beta = \alpha +1$ is odd}, \\
 0, \qquad & \text{if $\beta \geq \alpha +2$.}
\end{cases}
\end{equation}
\end{mylemma}

\subsection{The \cite{CFKRS} conjecture}
Suppose $F(d)$ is as in Theorem \ref{thm:mainthm}.  Then
\begin{myconj}[\cite{CFKRS}]
\label{conj}
Suppose $\alpha, \beta, \gamma$ lie in the rectangle $|\text{Re}(s)| \leq \frac{\varepsilon}{\log{X}}$, $|\text{Im}(s)| \leq X^{\varepsilon}$.  Then if $l= l_1 l_2$ is odd with $l_1$ squarefree and $l_2$ a square, then
\begin{multline}
\label{eq:conj}
\sumstar_{(d,2)=1} L(\thalf + \alpha, \chi_{8d}) L(\thalf + \beta, \chi_{8d}) L(\thalf + \gamma, \chi_{8d}) \chi_{8d}(l) F(d) 
\\
= \sum_{\epsilon_1, \epsilon_2, \epsilon_3 \in \{\pm 1 \}}  A_{\epsilon_1 \alpha, \epsilon_2 \beta,\epsilon_3 \gamma}(l) \Gamma_{\alpha, \beta, \gamma}^{\delta_1, \delta_2, \delta_3} \frac{ \widetilde{F}(1- \delta_1 \alpha - \delta_2 \beta - \delta_3 \gamma)}{2 \zeta_2(2) \sqrt{l_1}}  + O(X^{f} \sqrt{l_1} (l_1 l_2 X)^{\varepsilon}),
\end{multline}
where
\begin{multline}
\label{eq:Aoriginal}
 A_{\alpha,\beta,\gamma}(l) = \sum_{(n,2)=1} \frac{\sigma_{\alpha,\beta,\gamma}(l_1 n^2)}{n} \prod_{p | nl_1 l_2} (1 + p^{-1})^{-1} 
\\
= \zeta_2(1 + 2\alpha) \zeta_2(1 + 2\beta) \zeta_2(1 + 2\gamma) \zeta_2(1 + \alpha + \beta) \zeta_2(1 + \alpha + \gamma) \zeta_2(1 + \beta + \gamma) B_{\alpha,\beta,\gamma}(l),
\end{multline}
where $B_{\alpha, \beta, \gamma}$ has an absolutely convergent Euler product for the parameters in a neighborhood of the origin.  Furthermore, the meaning of $\Gamma_{\alpha,\beta,\gamma}^{\delta_1, \delta_2, \delta_3}$ is simply $\Gamma_{\alpha}^{\delta_1} \Gamma_{\beta}^{\delta_2} \Gamma_{\gamma}^{\delta_3}$, where $\delta_i = 0$ if $\epsilon_i = +1$, and $\delta_i = 1$ if $\epsilon_i = -1$.  
\end{myconj}
The conjecture is derived from the following orthogonality relation
\begin{equation}
\label{eq:orthogonality}
 \sumstar_{(d,2) = 1} \chi_{8d}(m) F(d) \sim \frac{\widetilde{F}(1)}{2 \zeta_2(2)} \prod_{p|m} (1+ p^{-1})^{-1},
\end{equation}
for $m$ an odd square, and is $o(X)$ otherwise (for fixed $m$).  This relation appears in a slightly different form in \cite{CFKRS} but can be directly calculated using Poisson summation along the lines of the calculations in Section \ref{section:MNcalculation}.

The expected value of $f$ is somewhat controversial.  The five authors \cite{CFKRS} conjectured that $f=1/2$ is allowable, while \cite{DGH} predict the existence of a main term with $f=3/4$; the subtlety here is that \cite{DGH} analyze the moments involving non-fundamental discriminants and correction factors at bad primes.  A sieving argument is required to convert between the different moments.
Q. Zhang \cite{Zhang} provided further evidence for $f=3/4$ and in fact gave a prediction for the numerical value of the constant in the lower-order term of size $X^{3/4}$.  The constant is $\approx -.2$ making the lower-order term difficult to detect with numerics; however, M. Alderson and M. Rubinstein \cite{AR} recently performed extensive calculations which perhaps show modest agreement with this (numerically small) lower-order term.

\section{Outline of the method}
Our recursive approach to the problem takes the form
\begin{mytheo}
\label{thm:recursive}
 If Conjecture \ref{conj} is true with $f > 3/4$, then it is true for $f$ replaced by $\frac34 + \frac{f-\frac34}{2f}$.
\end{mytheo}
The case $f=1$ is ``trivial'' in the sense that the error term is larger than the main term, but quite nontrivial in that it uses the quadratic large sieve of Heath-Brown \cite{H-B}.  
Actually, we could start with $f=7/4$ which is an immediate consequence of the convexity bound (however, we emphasize that we still require the quadratic large sieve in the proof of Theorem \ref{thm:recursive} though perhaps with extra work it could be removed).  In any event, the sequence $f_1 =1$, $f_2 = 7/8$, $f_3 = 23/28$, etc. leads to $f = 3/4$ as allowable (one can check that the sequence of $f_n$'s is decreasing and bounded below by $3/4$ and hence has a limit which is easily checked to be $3/4$).  

\subsection{Dissection}
Before embarking on the details of the proof, we shall give an overview of the whole argument, deferring proofs to later sections.  First we need some notation. Let $M(l) = M_{\alpha, \beta, \gamma}(l)$ be the moment on the left hand side of \eqref{eq:conj}.  Using the ``two-piece'' approximate functional equation from Proposition \ref{prop:AFE}, write
\begin{equation}
 M_{\alpha, \beta, \gamma}(l) = M_1(l) + M_{-1}(l),
\end{equation}
respective to the two sums in the approximate functional equation.  That is,
\begin{equation}
 M_1 = \sumstar_{(d,2) = 1} F(d) \sum_{n} \frac{\chi_{8d}(nl) \sigma_{\alpha,\beta,\gamma}(n)}{n^{\thalf}} V_{\alpha, \beta,\gamma} \left(\frac{n}{d^{3/2}}\right),
\end{equation}
and
\begin{equation}
\label{eq:M-1def}
 M_{-1} = \Gamma_{\alpha,\beta,\gamma} \sumstar_{(d,2) = 1} d^{-\alpha-\beta-\gamma} F(d)  \sum_{n} \frac{\chi_{8d}(nl) \sigma_{-\alpha,-\beta,-\gamma}(n)}{n^{\thalf}} V_{-\alpha, -\beta,-\gamma} \left(\frac{n}{d^{3/2}}\right).
\end{equation}
The expressions are similar enough that we may focus upon $M_1$ and then easily deduce analogous formulas for $M_{-1}$.  Indeed, $M_{-1}$ is the same as $M_1$ after swapping $\alpha$ and $-\alpha$, $\beta$ and $-\beta$, $\gamma$ and $-\gamma$, replacing $F(x)$ by $F_{-\alpha, -\beta, -\gamma}(x) = x^{-\alpha - \beta - \gamma} F(x)$, and multiplying by $\Gamma_{\alpha,\beta,\gamma}$, in that order.

As explained in the introduction, the plan is to use M\"{o}bius inversion on the sum over $d$ and to treat the resulting sum in two ways.
We begin by removing the condition that $d$ is squarefree, getting
\begin{equation}
\label{eq:M1expression}
 M_1= \sum_{(a,2l) = 1} \mu(a) \sum_{(d,2)=1} F(d a^2) \sum_{(n,2a)=1} \frac{\chi_{8d}(nl) \sigma_{\alpha,\beta,\gamma}(n)}{n^{\thalf}} V_{\alpha, \beta,\gamma} \left(\frac{n}{(a^2 d)^{3/2}}\right).
\end{equation}
Now we separate the terms with $a \leq Y$ and with $a > Y$ ($Y$ a parameter to be chosen later), writing $M_1 = M_N + M_R$, respectively.  We similarly write $M_{-1} = M_{-N} + M_{-R}$.

\subsection{Overview of $M_N$}
To evaluate $M_N$, we use Poisson summation, Lemma \ref{lemma:poisson}, on the sum over $d$, getting a sum over $k$, say.  The term $k=0$ gives a certain main term, which we denote $M_N(k=0)$.  As Soundararajan realized \cite{Sound}, the terms $k \neq 0$ should be separated into squares and non-squares.  The squares give three more main terms which we denote $M_N(k=\square, \alpha)$, $M_N(k=\square, \beta)$, $M_N(k=\square, \gamma)$.  Evaluating these ``off-diagonal'' main terms is much more subtle than for $k=0$ and is accomplished in Section \ref{section:MNwholesection} below.  The work with $M_N$ is summarized with the following
\begin{mylemma}
\label{lemma:MNresult}
For the special choice of $G(s)$ given by Remark \ref{remark:zero}, we have
\begin{equation}
 M_N = M_N(k=0) + M_N(k=\square, \alpha) + M_N(k=\square, \beta) + M_N(k=\square, \gamma) + O(Y l^{1/2+\varepsilon} X^{3/4+\varepsilon}).
\end{equation}
\end{mylemma}
The $k \neq 0$ terms can be naturally expressed as a certain contour integral; moving contours to the left picks up poles giving the three ``off-diagonal'' main terms.  Bounding the contour integral on the new lines of integration gives the error term.

\subsection{Overview of $M_R$}
The evaluation of $M_R$ is in some sense simpler than for $M_N$ as it is purely combinatorial.  However, the expressions become somewhat complicated; here we give a sketchy argument that we make rigorous in Section \ref{section:CalculatingMR} below.
Recall
\begin{equation}
M_R =  \sum_{\substack{(a,2l) = 1 \\ a > Y}} \mu(a) \sum_{(d,2)=1} F(d a^2) \sum_{(n,2a)=1} \frac{\chi_{8d}(nl) \sigma_{\alpha,\beta,\gamma}(n)}{n^{\thalf}} V_{\alpha, \beta,\gamma} \left(\frac{n}{(a^2 d)^{3/2}}\right).
\end{equation}
Now apply the change of variables $d \rightarrow b^2 d$ with the new $d$ squarefree, to get
\begin{equation}
\label{eq:MRdstar}
M_R = \sum_{\substack{(a,2l) = 1 \\ a > Y}} \mu(a) \sum_{(b,2l)=1}  \sumstar_{(d,2)=1} F(d (ab)^2) \sum_{(n,2ab)=1} \frac{\chi_{8d}(nl) \sigma_{\alpha,\beta,\gamma}(n)}{n^{\thalf}} V_{\alpha, \beta,\gamma} \left(\frac{n}{((ab)^2 d)^{3/2}}\right).
\end{equation}
Using the definition of $V$ as an integral representation \eqref{eq:Vdef}, we get that the inner sum over $n$ above is
\begin{equation}
\label{eq:ncoprime}
 \sum_{(n,2ab)=1} \frac{\chi_{8d}(nl) \sigma_{\alpha,\beta,\gamma}(n)}{n^{\thalf}} \frac{1}{2\pi i} \int_{(2)} \frac{G(s)}{s} g_{\alpha,\beta,\gamma}(s) \frac{((ab)^2 d)^{3s/2}}{n^s} ds
\end{equation}
%
%
This sum over $n$ almost produces a product of $L$-functions but the coprimality restriction $(n,2ab) = 1$ slightly perturbs it.  Ignoring this annoyance for the present discussion, we see that $M_R$ should be related to
\begin{equation}
\label{eq:recursivefake}
\sum_{\substack{(a,2l) = 1 \\ a > Y}} \mu(a) \sum_{(b,2l)=1} \frac{1}{2\pi i} \int_{(\varepsilon)} \frac{G(s)}{s} g_{\alpha,\beta,\gamma}(s) M_{\alpha+s, \beta+s, \gamma+s}(l) ds,
\end{equation}
where $M_{\alpha+s,\beta+s,\gamma+s}(l)$ is the same moment with which we started but with a different weight function having support with integers of size $X/(ab)^2$.  Then we can apply Conjecture \ref{conj} to this sum, expressing it as a sum of eight main terms (one for each choice of $\epsilon_1, \epsilon_2, \epsilon_3 \in \{\pm 1\}$), plus an error term.

After making these arguments rigorous (treating the coprimality restriction $(n, 2ab) = 1$, etc.), these ideas lead to
\begin{mylemma}  
\label{lemma:MRresult}
If Conjecture \ref{conj} holds with a parameter $f > 1/2$, then
\begin{equation}
\label{eq:MRresult}
 M_R = \sum_{\epsilon_1, \epsilon_2, \epsilon_3 \in \{\pm 1\}} M_R(\epsilon_1, \epsilon_2, \epsilon_3) + O(\sqrt{l} \frac{X^{f + \varepsilon}}{Y^{2f-1}}),
\end{equation}
where $M_R(\epsilon_1, \epsilon_2, \epsilon_3)$ is defined below as \eqref{eq:MRmainterms}.
\end{mylemma}

\subsection{How the main terms combine}
Of the eight main terms appearing in Lemma \ref{lemma:MRresult}, only four turn out to contribute significantly.  We suspect that finding an error better than $X^{3/4}$ would (at least) require dealing with all eight of these terms.  The smaller ones are bounded with the following
\begin{mylemma}
\label{lemma:MRbound}
 If exactly two of the $\epsilon_i$'s are $-1$, then for a special choice of $G(s)$ described in Remark \ref{remark:zero}, we have
\begin{equation}
\label{eq:twonegeps}
 M_R(\epsilon_1, \epsilon_2, \epsilon_3) \ll  Y X^{3/4} (l X)^{\varepsilon},
\end{equation}
and furthermore,
\begin{equation}
\label{eq:threenegeps}
 M_R(-1, -1, -1) \ll X^{3/4} (l X)^{\varepsilon}.
\end{equation}
\end{mylemma}
This dispenses with four of the terms on the right hand side of \eqref{eq:MRresult}, leaving four that do contribute.  The remaining ones combine in a very pleasant way with the four main terms of Lemma \ref{lemma:MNresult}.  The term $M_N(k=0)$ combines with $M_R(1,1,1)$, while $M_N(k=\square, *)$ with $* = \alpha, \beta$, $\gamma$ combines with $M_R(-1,1,1)$, $M_R(1,-1,1)$, $M_R(1,1,-1)$, respectively.  The combination is as pleasant as one could hope; each corresponding pair of $M_N$ and $M_R$ give almost identical expressions, the only difference being $M_N$ has $a \leq Y$ while $M_R$ has $a > Y$, so that combining them simply removes this restriction on the sum over $a$ altogether.  This matching of $M_N$ and $M_R$ is the only truly difficult part of this work, as it requires substantial calculation to see that these very different expressions actually agree.  Having done the simpler case of the first moment \cite{Y}, it was easier to predict that terms should combine in this way.

Each of the four combined pairs $M_N$ and $M_R$ is given as a certain contour integral.  By shifting the contour to the left, one crosses a pole at $s=0$ only; the residue at this point gives one of the eight main terms on the right hand side of \eqref{eq:conj}.  The integral along the new contour is estimated with absolute values.  The end result is
\begin{mylemma}
\label{lemma:MNk0andMR111}
For a special choice of $G(s)$ described in Remark \ref{remark:zero}, we have
\begin{equation}
\label{eq:MNk0andMR111}
 M_N(k=0) + M_R(1,1,1) = A_{\alpha, \beta, \gamma}(l) \frac{\widetilde{F}(1)}{2 \zeta_2(2) \sqrt{l_1}} + O(X^{1/4 + \varepsilon} l^{\varepsilon}),
\end{equation}
and furthermore,
\begin{equation}
\label{eq:MNksquareandMR-111}
 M_N(k=\square, \alpha) + M_R(-1,1,1) = A_{- \alpha, \beta,\gamma}(l) \Gamma_{\alpha} \frac{ \widetilde{F}(1- \alpha )}{2 \zeta_2(2) \sqrt{l_1}} + O(X^{3/4 + \varepsilon} l^{\varepsilon}),
\end{equation}
and similarly for the other two terms (with $\alpha$ replaced by either $\beta$ or $\gamma$).
\end{mylemma}
By combining Lemmas \ref{lemma:MNresult}, \ref{lemma:MRresult}, \ref{lemma:MRbound}, and \ref{lemma:MNk0andMR111}, we get that $M_1$ equals the sum of the four main terms on the right hand side of \eqref{eq:conj}, namely those with at most one of the $\epsilon_i = -1$, plus an error of size
\begin{equation}
\label{eq:M1error}
 \ll l^{1/2 + \varepsilon} X^{3/4 + \varepsilon} Y + l^{1/2 + \varepsilon} \frac{X^{f + \varepsilon}}{Y^{2f-1}}.
\end{equation}
We then obtain an asymptotic for $M_{-1}$ by the procedure described following \eqref{eq:M-1def}.  An examination of the form of the main terms on the right hand side of \eqref{eq:conj} shows that $M_{-1}$ accounts for the remaining four main terms in \eqref{eq:conj}, namely those with at least two of the $\epsilon_i = -1$, plus the same error as given in \eqref{eq:M1error}.  Thus, taken together, these results show
\begin{mylemma}
\label{lemma:selfreference}
If Conjecture \ref{conj} holds with some $f > 1/2$, then 
 \begin{equation}
  M_1 + M_{-1} = M.T. + O(l^{1/2 + \varepsilon} X^{3/4 + \varepsilon} Y) + O(\frac{X^{f + \varepsilon}}{Y^{2f-1}}),
 \end{equation}
where $M.T.$ is the main term on the right hand side of \eqref{eq:conj}.
\end{mylemma}
Choosing $Y = X^{\frac{f-\frac34}{2f}}$ gives Theorem \ref{thm:recursive}.
It has been so far left implicit that the error terms in the above lemmas are uniform with respect to the shift parameters lying in the appropriate punctured rectangles described in Section \ref{section:AFE}.  In Lemma \ref{lemma:selfreference} we can claim the same error term for shift parameters in the non-punctured rectangles $|\text{Re}(\lambda)| \leq c_1/\log{X}$, $|\text{Im}(\lambda) \leq c_2 X^{\varepsilon}$ by the following reasoning.  We note that $M_1 + M_{-1}$ is analytic in terms of the shift parameters in this larger domain; by work of \cite{CFKRS}, the main term in \eqref{eq:conj} has this same property.  Therefore, so does the error term.  The maximum modulus principle shows the error term is maximized in the punctured rectangle domain where we initially proved the result, and so the uniformity is extended to the larger domain as claimed.

We prove Lemma \ref{lemma:MRresult} in Section \ref{section:proofofMRresult}, Lemma \ref{lemma:MRbound} in Section \ref{section:proofofMRbound}, Lemma \ref{lemma:MNresult} in Section \ref{section:proofofMNresult}, and Lemma \ref{lemma:MNk0andMR111} in Section \ref{section:proofofMNk0andMR111}.  Taking these four results for granted, we already showed that Lemma \ref{lemma:selfreference} is valid, and hence Theorem \ref{thm:recursive} holds.  The five authors \cite{CFKRS} clearly explain how to let the shift parameters tend to zero and hence deduce Theorem \ref{thm:mainthm} from Theorem \ref{thm:recursive}.

\section{Calculating with $M_R$}
\label{section:CalculatingMR}
In this section we obtain Lemma \ref{lemma:MRresult} and Lemma \ref{lemma:MRbound}.  

\subsection{Proof of Lemma \ref{lemma:MRresult}}
\label{section:proofofMRresult}
\begin{proof}
Our first subsidiary goal is to find the rigorous analog of \eqref{eq:recursivefake}.  This formula is
\begin{multline}
\label{eq:MRcoprime}
M_R = \sum_{\substack{(a,2l) = 1 \\ a > Y}}  \mu(a)  \sum_{(b,2l)=1} \sum_{r_1, r_2, r_3 | ab} \frac{\mu(r_1) \mu(r_2) \mu(r_3) }{r_1^{\half + \alpha} r_2^{\half + \beta} r_3^{\half + \gamma}} \frac{1}{2\pi i} \int_{(\varepsilon)} \sumstar_{(d,2)=1}  F(d a^2 b^2) \chi_{8d}(lr_1 r_2 r_3) 
\\
  L(\tfrac12 + \alpha + s, \chi_{8d}) L(\tfrac12 + \beta + s, \chi_{8d}) L(\tfrac12 + \gamma + s, \chi_{8d}) (d a^2 b^2)^{3s/2} (r_1 r_2 r_3)^{-s} \frac{G(s)}{s} g_{\alpha,\beta,\gamma}(s) ds.
\end{multline}
To prove this, we pick up with \eqref{eq:MRdstar}, recalling \eqref{eq:ncoprime} is the sum over $n$.  Then note
\begin{multline}
\label{eq:ncoprimeformula}
 \sum_{(n,2ab)=1} \frac{\chi_{8d}(n) \sigma_{\alpha,\beta,\gamma}(n)}{n^s} = \prod_{p \nmid 2ab} 
( 1- \frac{\chi_{8d}(p)}{p^{\alpha + s}})^{-1} 
( 1- \frac{\chi_{8d}(p)}{p^{\beta + s}})^{-1}
( 1- \frac{\chi_{8d}(p)}{p^{\gamma + s}})^{-1}
\\
=
L(\alpha + s, \chi_{8d}) L(\beta + s, \chi_{8d}) L(\gamma + s, \chi_{8d}) \sum_{r_1, r_2, r_3 | ab} \frac{\mu(r_1) \mu(r_2) \mu(r_3) \chi_{8d}(r_1 r_2 r_3)}{r_1^{\alpha + s} r_2^{\beta + s} r_3^{\gamma + s}}.
\end{multline}
Inserting \eqref{eq:ncoprimeformula} into \eqref{eq:ncoprime}, and then into \eqref{eq:MRdstar}, gives \eqref{eq:MRcoprime}, as desired.  We were able to move the line of integration to $\sigma = \varepsilon$ since the Dirichlet $L$-functions have no poles.  We take $\varepsilon = 1/\log{X}$.

Now we proceed directly to the proof of Lemma \ref{lemma:MRresult}.  Note that the inner sum over $d$ appearing in \eqref{eq:MRcoprime}, namely
\begin{equation}
 \sumstar_{(d,2)=1}  \chi_{8d}(lr_1 r_2 r_3) L(\tfrac12 + \alpha + s, \chi_{8d}) L(\tfrac12 + \beta + s, \chi_{8d}) L(\tfrac12 + \gamma + s, \chi_{8d}) (d a^2 b^2)^{3s/2}  F(d a^2 b^2),
\end{equation}
is of the form $M_{\alpha+s,\beta+s, \gamma+s}(lr_1 r_2 r_3)$, but with a new weight function with smaller support ($d \asymp X/(ab)^2$).
Write
\begin{equation}
 (da^2 b^2)^{3s/2}  F(d a^2 b^2) =  F_{\frac{3s}{2};a^2b^2}(d),
\end{equation}
where $F_{\nu;y}(x) = (xy)^{\nu} F(xy)$.  

Next we apply Conjecture \ref{conj} to this inner sum over $d$.  Technically, we need to first truncate the $s$-integral so that $|\text{Im}(s)| \leq (\log(X/a^2b^2))^2$.  When $(ab)^2 \leq X^{1-\varepsilon}$, then the exponential decay of the integrand shows that the error incurred by this truncation is negligible ($\ll X^{-100}$, say).  Otherwise, when $(ab)^2 \geq X^{1-\varepsilon}$ then the sum over $d$ is almost bounded so that the convexity bound $L(1/2 + \alpha + s, \chi_{8d}) \ll ((1+|s|)|d|)^{1/4}$ is sufficient to show these terms give $O(X^{\half + \varepsilon})$.  After plugging in Conjecture \ref{conj} to the truncated integral, then the same argument works to extend the integral back to the whole vertical line, without introducing a new error.  By this procedure, we get $M_R$ as the sum of eight main terms plus an error of size
\begin{equation}
 \ll \sum_{\substack{a > Y}} |\mu(a)| \sum_{(b,2)=1}  \sum_{r_1, r_2, r_3 | ab} \frac{|\mu(r_1) \mu(r_2) \mu(r_3)|}{(r_1 r_2 r_3)^{\half}} \sqrt{l r_1 r_2 r_3} \leg{X}{a^2 b^2}^{f + \varepsilon} \ll \sqrt{l} \frac{X^{f + \varepsilon}}{Y^{2f - 1}} l^{\varepsilon}.
\end{equation}
This is the error term appearing in Lemma \ref{lemma:MRresult}.

As for the main terms, we have by a direct application of Conjecture \ref{conj} that
\begin{multline}
M_R(\epsilon_1, \epsilon_2, \epsilon_3) = \sum_{\substack{(a,2l) = 1 \\ a > Y}} \mu(a)  \sum_{(b,2l)=1} \sum_{r_1, r_2, r_3 | ab} \frac{\mu(r_1) \mu(r_2) \mu(r_3) }{r_1^{\half+\alpha} r_2^{\half + \beta} r_3^{\half + \gamma}} \frac{1}{\sqrt{(lr_1 r_2 r_3)^*}} \frac{1}{2 \zeta_2(2)}
\\
\frac{1}{2\pi i} \int_{(\varepsilon)} A_{\epsilon_1(\alpha + s), \epsilon_2(\beta + s), \epsilon_3(\gamma + s)}(lr_1 r_2 r_3) \Gamma_{\alpha + s, \beta + s, \gamma + s}^{\delta_1, \delta_2, \delta_3}
\\
\widetilde{F}_{3s/2;a^2 b^2}(1 - \delta_1(\alpha+s) - \delta_2(\beta + s) - \delta_3(\gamma + s)) \frac{1}{(r_1 r_2 r_3)^{s}} \frac{G(s)}{s} g_{\alpha,\beta,\gamma}(s) ds.
\end{multline}
Here we use the notation $m^* = m_1$, where $m=m_1 m_2$ where $m_1$ is squarefree and $m_2$ is a square.
We can simplify the Mellin transform a bit by noting
\begin{equation}
 \widetilde{F}_{3s/2;a^2 b^2}(u) = \int_0^{\infty} (xa^2 b^2)^{\frac{3s}{2}} F(xa^2 b^2) x^{u} \frac{dx}{x} = (ab)^{-2u} \widetilde{F}(\tfrac{3s}{2} + u).
\end{equation}
As shorthand, let
$w = 1 - \delta_1(\alpha+s) - \delta_2(\beta+s) - \delta_3(\gamma+s)$.
Then
\begin{multline}
\label{eq:MRmainterms}
M_R(\epsilon_1, \epsilon_2, \epsilon_3) =  \frac{1}{2 \zeta_2(2)} \sum_{\substack{(a,2l) = 1 \\ a > Y}} \mu(a) 
\frac{1}{2\pi i} \int_{(\varepsilon)} \Gamma_{\alpha + s, \beta + s, \gamma + s}^{\delta_1, \delta_2, \delta_3}
\frac{G(s)}{s} g_{\alpha,\beta,\gamma}(s) 
\widetilde{F}(\tfrac{3s}{2} + w)
\\
 \sum_{\substack{(b,2l)=1}} \frac{1}{(ab)^{2w}} \sum_{r_1, r_2, r_3 | ab} \frac{\mu(r_1) \mu(r_2) \mu(r_3)}{r_1^{\half +\alpha + s}r_2^{\half +\beta + s}r_3^{\half +\gamma + s}} \frac{1}{\sqrt{(lr_1 r_2 r_3)^*}}
 A_{\epsilon_1(\alpha + s), \epsilon_2(\beta + s), \epsilon_3(\gamma + s)}(lr_1 r_2 r_3)   
  ds.
\end{multline}
These are the main terms appearing in Lemma \ref{lemma:MRresult}, completing the proof.
\end{proof}

\subsection{Proof of Lemma \ref{lemma:MRbound}}
\label{section:proofofMRbound}
 Here we prove Lemma \ref{lemma:MRbound}.  We shall need an auxiliary result that we first motivate.
By a symmetry argument, for the proof of \eqref{eq:twonegeps} it suffices to consider the case $\epsilon_1 = \epsilon_2 = -1$, $\epsilon_3 = 1$.  In this case, we have
\begin{multline}
M_R(-1, -1, 1) =  \frac{1}{2 \zeta_2(2)} \sum_{\substack{(a,2l) = 1 \\ a > Y}} \mu(a) 
\frac{1}{2\pi i} \int_{(\varepsilon)} \Gamma_{\alpha + s} \Gamma_{\beta + s}
\frac{G(s)}{s} g_{\alpha,\beta,\gamma}(s) 
\widetilde{F}(1-\alpha-\beta-\tfrac{s}{2})
\\
 \sum_{\substack{(b,2l)=1}} \frac{1}{(ab)^{2-2\alpha-2\beta-4s}} \sum_{r_1, r_2, r_3 | ab} \frac{\mu(r_1) \mu(r_2) \mu(r_3)}{r_1^{\half +\alpha + s}r_2^{\half +\beta + s}r_3^{\half +\gamma + s}} \frac{1}{\sqrt{(lr_1 r_2 r_3)^*}}
 A_{-\alpha - s, -\beta - s, \gamma + s}(lr_1 r_2 r_3)   
  ds.
\end{multline}
The important feature here is that the weight function satisfies $\widetilde{F}(1-\alpha-\beta-\tfrac{s}{2}) \ll X^{1-\tfrac{\sigma}{2} + \varepsilon}$, where $\sigma = \text{Re}(s)$ so that moving the contour to the {\em right} gives a saving.  
The hitch in carrying this out is getting an analytic continuation of the integrand for sufficiently large $\sigma$.  For $\sigma$ large the sum over $a$ will not converge absolutely; to get around this, we extend $a$ to all positive integers, and subtract the contribution from $a \leq Y$, writing $M_R(-1,-1,1) = M'(-1,-1,1) - M''(-1,-1,1)$, respectively.

First we simplify $M'(-1,-1,1)$; by grouping $ab$ into a new variable, say $c$, then the sum over $a$ becomes $\sum_{a | c} \mu(a)$, whence $c=1$ is the only term that does not vanish.  Then $b=1$, and $r_1 = r_2 =r_3 =1$, so that
\begin{multline}
\label{eq:M'def}
M'(-1, -1, 1) =  \frac{1}{2 \zeta_2(2)} 
\frac{1}{2\pi i} \int_{(\varepsilon)} \Gamma_{\alpha + s} \Gamma_{\beta + s}
\frac{G(s)}{s} g_{\alpha,\beta,\gamma}(s) 
\widetilde{F}(1-\alpha-\beta-\tfrac{s}{2})
\\
 \frac{1}{\sqrt{l_1}}
 A_{-\alpha - s, -\beta - s, \gamma + s}(l)   
  ds.
\end{multline}

As for $M''(-1,-1,1)$, we sum over $b$ to get
\begin{multline}
\label{eq:M''def}
M''(-1, -1, 1) =  \frac{1}{2 \zeta_2(2)} \sum_{\substack{(a,2l) = 1 \\ a \leq Y}} \mu(a) 
\frac{1}{2\pi i} \int_{(\varepsilon)} \Gamma_{\alpha + s} \Gamma_{\beta + s}
\frac{G(s)}{s} g_{\alpha,\beta,\gamma}(s) 
\widetilde{F}(1-\alpha-\beta-\tfrac{s}{2})
\\
 \frac{1}{a^{2-2\alpha-2\beta-4s}} \sum_{(r_1 r_2 r_3, 2l)=1} \frac{\mu(r_1) \mu(r_2) \mu(r_3)}{r_1^{\half +\alpha + s}r_2^{\half +\beta + s}r_3^{\half +\gamma + s}}\leg{(a,[r_1, r_2, r_3])}{[r_1, r_2, r_3]}^{2-2\alpha-2\beta-4s} 
\\
\zeta_{2l}(2-2\alpha-2\beta - 4s) \frac{1}{\sqrt{(lr_1 r_2 r_3)^*}}
 A_{-\alpha - s, -\beta - s, \gamma + s}(lr_1 r_2 r_3)   
  ds.
\end{multline}

We shall estimate both \eqref{eq:M'def} and \eqref{eq:M''def} in essentially the same way, namely by moving the contour of integration to $\sigma = \half - \varepsilon$ and bounding everything with absolute values.  We shall show that the sums over the $r_i$ converge absolutely.  In this way we reduce the problem to obtaining the analytic behavior of the arithmetical factor to this region.
This behavior is detailed with the following

\begin{mylemma}
 \label{lemma:A}
 Let $l = l_1 l_2$ and $A_{\alpha, \beta, \gamma}(l)$ be as in Conjecture \ref{conj}.  Then $A_{\alpha, \beta, \gamma}(l)$ has a meromorphic continuation to $\text{Re}(\alpha), \text{Re}(\beta), \text{Re}(\gamma) > -\half$.  Precisely, for any positive integer $M$ there exist integers $d_{a,b,c}$ (possibly negative or zero) such that
\begin{equation}
\label{eq:Aproduct}
 A_{\alpha,\beta,\gamma}(l) = C_{\alpha,\beta,\gamma}(l)\prod_{\substack{a + b + c \leq M-1 \\ a, b, c \geq 0}} \zeta(a+b+c + 2a\alpha + 2b\beta + 2c\gamma)^{d_{a,b,c}} ,
\end{equation}
where for any $\delta > 0$, $C_{\alpha,\beta,\gamma}(l)$ is given by an absolutely convergent Euler product in the region $\text{Re}(\alpha), \text{Re}(\beta), \text{Re}(\gamma) > \frac{1-M}{2M} + \delta$.  Furthermore, in this region $C_{\alpha, \beta, \gamma}(l)$ satisfies the bound
\begin{equation}
 C_{\alpha,\beta,\gamma}(l) \ll \sqrt{l_1} l^{\varepsilon}.
\end{equation}
\end{mylemma}
\noindent {\bf Remark.}  If $a + b + c =1$ then $d_{a,b,c} = 1$, which justifies the representation on the second line of \eqref{eq:Aoriginal}.

Before proving Lemma \ref{lemma:A}, we conclude the proof of Lemma \ref{lemma:MRbound}.  
We require the estimate
\begin{equation}
 \widetilde{F}(1-\alpha-\beta- \tfrac{s}{2}) \ll X^{1 - \tfrac{\sigma}{2}},
\end{equation}
and that $G(s)$ is prescribed according to Remark \ref{remark:zero}.
Thus by moving the contour to $\sigma = \half - \delta$, we have
\begin{equation}
 M'(-1,-1,1) \ll l_1^{-\half} l_1^{\half+\varepsilon} l^{\varepsilon} X^{1-1/4 + \varepsilon},
\end{equation}
and
\begin{equation}
 M''(-1,-1,1) \ll X^{3/4 + \delta/2} \sum_{a \leq Y} a^{4\delta} \sum_{r_1, r_2, r_3} \frac{(lr_1 r_2 r_3)^{\varepsilon} |\mu(r_1) \mu(r_2) \mu(r_3)|}{(r_1 r_2 r_3)^{1- \delta}} \leg{(a,[r_1, r_2, r_3])}{[r_1, r_2, r_3]}^{4\delta}.
\end{equation}
Here $\varepsilon > 0$ is arbitrarily small.  Using the trivial inequality $(a, [r_1, r_2, r_3]) \leq a$ and noting that the sums over $r_1, r_2, r_3$ converge absolutely for any $\delta > 0$, (provided $\varepsilon < \delta/2$, say).  Thus by taking $\delta > 0$ arbitrarily small, we get
\begin{equation}
 M''(-1,-1,1) \ll Y X^{3/4} (l X)^{\varepsilon}.
\end{equation}
These bounds furnish the proof of \eqref{eq:twonegeps}.  The case of \eqref{eq:threenegeps} is similar but easier.  In this case,
\begin{multline}
M_R(-1, -1, -1) =  \frac{1}{2 \zeta_2(2)} \sum_{\substack{(a,2l) = 1 \\ a > Y}} \mu(a) 
\frac{1}{2\pi i} \int_{(\varepsilon)} \Gamma_{\alpha + s, \beta + s, \gamma + s}
\frac{G(s)}{s} g_{\alpha,\beta,\gamma}(s) 
\widetilde{F}(1- \alpha-\beta-\gamma-\tfrac{3s}{2})
\\
 \sum_{\substack{(b,2l)=1}} \frac{1}{(ab)^{2(1-\alpha-\beta-\gamma-3s)}} \sum_{r_1, r_2, r_3 | ab} \frac{\mu(r_1) \mu(r_2) \mu(r_3)}{r_1^{\half +\alpha + s}r_2^{\half +\beta + s}r_3^{\half +\gamma + s}} \frac{A_{-\alpha - s, -\beta - s, -\gamma - s}(lr_1 r_2 r_3)}{\sqrt{(lr_1 r_2 r_3)^*}}
  ds.
\end{multline}
In this case we simply move $\text{Re}(s) = \sigma$ to $\sigma = \frac16 - \delta$ for $\delta$ small.  The sums over $a$ and $b$ converge absolutely so an easy argument gives \eqref{eq:threenegeps}.

\begin{proof}[Proof of Lemma \ref{lemma:A}]
Recall
\begin{equation}
 A_{\alpha,\beta,\gamma}(l) = \sum_{(n,2)=1} \frac{\sigma_{\alpha,\beta,\gamma}(l_1 n^2)}{n} \prod_{p | nl} (1 + p^{-1})^{-1},
\end{equation}
which is initially defined for $\text{Re}(\alpha, \beta, \gamma) > 0$. 
The function $A$ has an Euler product representation, say $A=\prod_p A_p$.  Write $l_1 = \prod_p p^{l_p}$, say.  Then
\begin{equation}
 A_p = \sum_{j=0}^{\infty} \frac{\sigma_{\alpha,\beta,\gamma}(p^{l_p+2j})}{p^j} \prod_{q | lp^j} (1 + q^{-1})^{-1}.
\end{equation}
The $(1+q^{-1})^{-1}$ term is mildly annoying and is benign with respect to proving Lemma \ref{lemma:A}, so we first make a simple approximation to remove this factor.
For $p \nmid 2l$, $A_p$ has the form
\begin{equation}
A_p= 1 + (1+p^{-1})^{-1} \sum_{j=1}^{\infty} \frac{\sigma_{\alpha,\beta,\gamma}(p^{2j})}{p^j}
= 1 + (1+p^{-1})^{-1}(-1 + A_p'),
. 
\end{equation}
say, where
\begin{equation}
 A_p' = \sum_{j=0}^{\infty} \frac{\sigma_{\alpha,\beta,\gamma}(p^{2j})}{p^j}.
\end{equation}
Note that for $\text{Re}(\alpha), \text{Re}(\beta), \text{Re}(\gamma) > -\half$,
\begin{equation}
 A_p' = 1 + O(p^{-1}(p^{-2\alpha} + p^{-2\beta} + p^{-2\gamma} + p^{-\alpha-\beta} + p^{-\alpha-\gamma} + p^{-\beta-\gamma})).
\end{equation}
Thus
\begin{equation}
 \frac{A_p}{A_p'} = 1 + \frac{1}{p} \frac{(-1+A_p')}{A_p'} = 1 + O(p^{-2}(p^{-2\alpha} + p^{-2\beta} + p^{-2\gamma} + p^{-\alpha-\beta} + p^{-\alpha-\gamma} + p^{-\beta-\gamma})),
\end{equation}
so that
 $\prod_{p \nmid 2l} \frac{A_p}{A_p'}$
is absolutely convergent in any domain of the form $\text{Re}(\alpha)$, $\text{Re}(\beta)$, $\text{Re}(\gamma) \geq -\half + \delta$ with $\delta > 0$ fixed.  Thus $\prod_{p \nmid 2l} A_p$ has the meromorphy properties required by Lemma \ref{lemma:A} if and only if $\prod_{p \nmid 2l} A_p'$ does.  (We shall treat the primes dividing $l$ later).

Note that
\begin{equation}
 \sigma_{\alpha, \beta, \gamma}(p^k) = \sum_{a + b + c=k} p^{-a\alpha - b \beta - c \gamma}.
\end{equation}
Thus for $p \nmid 2l$,
\begin{equation}
A_p'= \sum_{j=0}^{\infty} \frac{\sigma_{\alpha,\beta,\gamma}(p^{2j})}{p^j} = \sum_{a, b, c \geq 0} \half\frac{1 + (-1)^{a+b+c}}{p^{\frac{a+b+c}{2} + a\alpha + b\beta + c\gamma}}.
\end{equation}
Let $x = p^{-\half - \alpha}$, $y=p^{-\half-\beta}$, and $z = p^{-\half-\gamma}$.  Then with this notation,
\begin{equation}
A_p'= \sum_{j=0}^{\infty} \frac{\sigma_{\alpha,\beta,\gamma}(p^{2j})}{p^j} = \sum_{a, b, c \geq 0} \half (1 + (-1)^{a+b+c}) x^a y^b z^c,
\end{equation} 
which simplifies as
\begin{equation}
A_p'= \half (\frac{1}{(1-x)(1-y)(1-z)} + \frac{1}{(1+x)(1+y)(1+z)}) = \frac{1 + xy + xz + yz}{(1-x^2)(1-y^2)(1-z^2)}.
\end{equation}
Let $D(x,y,z) = (1-x^2)(1-y^2)(1-z^2)(1-xy)(1-xz)(1-yz)$, whence 
\begin{equation}
A_p'= \frac{N(x,y,z)}{D(x,y,z)} 
\end{equation}
where
\begin{equation}
 N(x,y,z) = 1-x^2 y^2 -x^2 z^2 - y^2 z^2 - x^2 yz - xy^2 z  - xyz^2  + (\text{degree 6 and higher}).
\end{equation}
Furthermore, $N(x,y,z) \in \mz[x,y,z]$.
Then we can write
\begin{equation}
 N(x,y,z) = (1-x^2 y^2)(1-x^2 z^2)(1-y^2 z^2)(1-x^2 yz)(1-xy^2 z)(1- x y z^2)N_2(x,y,z)
\end{equation}
where $N_2(x,y,z) = 1 + (\text{degree 6 and higher})$.  It is clear that this process can be continued indefinitely to obtain $A_p'$ as a ratio of products of polynomials of the form $(1-x^i y^j z^k)$ with $i +j + k$ less than any given bound, say $B$, times a polynomial of the form $1$ plus a symmetric polynomial of degree $\geq B+2$.  In terms of $\prod_{p \nmid 2l} A_p'$, this gives a factorization in terms of ratios of the zeta function, of the form
\begin{multline}
\label{eq:Ap'product}
\prod_{p \nmid 2l} A_p' = \frac{\zeta_{2l}(1 + 2\alpha )\zeta_{2l}(1 + 2\beta )\zeta_{2l}(1 + 2\gamma)\zeta_{2l}}{\zeta_{2l}(2 + 2\alpha + 2\beta)\zeta_{2l}(2 + 2\alpha + 2\gamma) \zeta_{2l}(2 + 2\beta + 2\gamma)}
\\
\frac{\zeta_{2l}(1 + \alpha + \beta )\zeta_{2l}(1 + \alpha + \gamma )\zeta_{2l}(1 + \beta+\gamma)}{\zeta_{2l}(2 + 2\alpha + \beta+\gamma)\zeta_{2l}(2 + \alpha + 2\beta+\gamma)\zeta_{2l}(2 + \alpha + \beta+2\gamma)} \dots,
\end{multline}
where the dots indicate the contribution from the terms corresponding to the degree $6$ factors of $N_2(x,y,z)$.  The degree $6$ factors give an absolutely convergent Euler product for $\text{Re}(\alpha), \text{Re}(\beta), \text{Re}(\gamma) > -\frac{1}{3}$.  By continuing this procedure, we continue extracting ratios of zeta functions, each of the form $\zeta_{2l}(M + a\alpha + b\beta + c\gamma)$ with $a+b+c=2M$; by taking such terms corresponding to degree $\leq M-1$ polynomials of $x,y,z$, we get absolute convergence for the ``remainder'' term provided $\text{Re}(\alpha), \text{Re}(\beta), \text{Re}(\gamma) > \frac{1-M}{2M}$ which as $M \rightarrow \infty$ leads to $\text{Re}(\alpha), \text{Re}(\beta), \text{Re}(\gamma) > -\frac{1}{2}$, as desired.

As for $p | l$, first notice that extending the product in \eqref{eq:Ap'product} to $p | l$ gives the ratio of zeta functions as in \eqref{eq:Aproduct}, and that this finite product is holomorphic and bounded by $l^{\varepsilon}$ in the region $\text{Re}(\alpha), \text{Re}(\beta), \text{Re}(\gamma) \geq -\half + \delta$.

To complete the proof of Lemma \ref{lemma:A}, we need to examine the behavior of $\prod_{p | l} A_p$.  For $p |l$, 
\begin{equation}
 A_p = (1+p^{-1})^{-1} \sum_{j=0}^{\infty} \frac{\sigma_{\alpha,\beta,\gamma}(p^{2j+l_p})}{p^j}.
\end{equation}
Here $l_p = 0$ or $1$.  The work above easily handles the case $l_p =0$, so suppose $l_p = 1$.  Then
\begin{equation}
\label{eq:sigmasum}
 \sum_{j=0}^{\infty} \frac{\sigma_{\alpha,\beta,\gamma}(p^{2j+1})}{p^j} = \sum_{a + b + c \equiv 1 \shortmod{2}} p^{-a\alpha - b \beta - c \gamma - \frac{a+b+c-1}{2}},
\end{equation}
which with the notation $x,y,z$ as before, equals
\begin{equation}
 \sqrt{p} \sum_{a,b,c} \half (1- (-1)^{a+b+c}) x^a y^b z^c = \sqrt{p} \half (\frac{1}{(1-x)(1-y)(1-z)} - \frac{1}{(1+x)(1+y)(1+z)}).
\end{equation}
Simplifying this, we get that \eqref{eq:sigmasum} equals
\begin{equation}
\sqrt{p} \frac{x+y+z + xyz}{(1-x^2)(1-y^2)(1-z^2)}.
\end{equation}
For $\text{Re}(\alpha, \beta, \gamma) \geq -\half + \delta> -\half$ this is analytic and satisfies the bound
\begin{equation}
 \ll |p^{-\alpha}| + |p^{-\beta}| + |p^{-\gamma}| \ll p^{1/2},
\end{equation}
with an implied constant depending on $\delta$.  Thus $\prod_{p|l} A_p$ is holomorphic in the desired region, and satisfies
\begin{equation}
 \prod_{p | l} A_p \ll l_1^{\half} l^{\varepsilon},
\end{equation}
as desired.  This completes the proof of Lemma \ref{lemma:A}.
\end{proof}

\section{Development of $M_N$}
\label{section:MNwholesection}

\subsection{Poisson summation}
\label{section:MNcalculation}
We shall use Poisson summation to analyze $M_N$.  We perform these calculations here.  Recall that $M_N$ is given by \eqref{eq:M1expression} but with the additional truncation condition $a \leq Y$ imposed.  We write
\begin{equation}
 M_N = \sum_{\substack{(a,2l) = 1 \\ a \leq Y}} \mu(a)  \sum_{(n,2a)=1} \leg{8}{nl} \frac{\sigma_{\alpha,\beta,\gamma}(n)}{n^{\thalf}} \sum_{(d,2)=1} \leg{d}{nl} F(d a^2) V_{\alpha, \beta,\gamma} \left(\frac{n}{(a^2 d)^{3/2}}\right).
\end{equation}
By Poisson summation, Lemma \ref{lemma:poisson}, this is
\begin{equation}
\label{eq:MNafterPoisson}
 M_N = \frac{1}{2} \sum_{\substack{(a,2l) = 1 \\ a \leq Y}} \mu(a)  \sum_{(n,2a)=1} \leg{16}{nl} \frac{\sigma_{\alpha,\beta,\gamma}(n)}{n^{\thalf}}  \sum_{k \in \mz} (-1)^k \frac{G_k(nl)}{nl} H(k/2nl),
\end{equation}
where
\begin{equation}
 H(y) = \int_{0}^{\infty} (\cos + \sin)(2 \pi xy) F(x a^2) V_{\alpha, \beta,\gamma} \left(\frac{n}{(a^2 x)^{3/2}}\right) dx.
\end{equation}
We find an alternate Mellin-type expression for $H$ with the following.
\begin{mylemma}
\label{lemma:H}
 Suppose that $c_s, c_u$ are real numbers such that $ \frac{3c_s}{2} - c_u > 0$ and $c_s > 0$.  Then for $y \neq 0$, we have
\begin{multline}
\label{eq:H}
 H(y) = \frac{1}{a^2} \leg{1}{2 \pi i}^2 \int_{(c_s)} \int_{(c_u)} \widetilde{F}(1+u) \frac{G(s)}{s} g_{\alpha, \beta, \gamma}(s) \leg{a^2}{2 \pi |y|}^{\frac{3s}{2} - u} n^{-s} 
\\
\Gamma(\frac{3s}{2} - u) (\cos + \sgn(y) \sin)(\frac{\pi}{2}(\frac{3s}{2} - u))  ds du.
\end{multline}
Furthermore, for $c_s > 0$, we have
\begin{equation}
 H(0) = \frac{1}{a^2} \frac{1}{2 \pi i} \int_{(c_s)} \frac{G(s)}{s} g_{\alpha, \beta, \gamma}(s) n^{-s} \widetilde{F}(1+\frac{3s}{2}) ds.
\end{equation}
\end{mylemma}
\begin{proof}
 First suppose $y \neq 0$.  We begin by writing
\begin{equation}
 F(xa^2) = \frac{1}{2 \pi i} \int_{(c_u)} \widetilde{F}(1+u) (xa^2)^{-1-u} du,
\end{equation}
valid for any $c_u \in \mr$, and the Mellin formula for $V_{\alpha, \beta, \gamma}$, i.e., \eqref{eq:Vdef}.  Then we change variables $x \rightarrow x/(2 \pi |y|)$ to obtain
\begin{multline}
 H(y) = \frac{1}{a^2} \int_{0}^{\infty} (\cos + \sgn(y) \sin)(x) \leg{1}{2 \pi i}^2 \int_{(c_s)} \int_{(c_u)} \widetilde{F}(1+u) \frac{G(s)}{s} g_{\alpha, \beta, \gamma}(s)   
\\
n^{-s} \leg{a^2}{2 \pi |y|}^{\frac{3s}{2} - u} x^{\frac{3s}{2} - u} ds du \frac{dx}{x}.
\end{multline}
Next we reverse the order of integrals in order to perform the $x$-integral first.  We temporarily make the assumption $0 < \frac{3 c_s}{2} - c_u < 1$ for convergence.  We argue as in my previous paper to justify the reversal of integrals.  Then using
\begin{equation}
 \int_0^{\infty} CS(x) x^{w} \frac{dx}{x} = \Gamma(w) CS(\frac{\pi w}{2}),
\end{equation}
where $CS$ stands for either $\cos$ or $\sin$, and $0 < \text{Re}(w) < 1$, we obtain \eqref{eq:H}.

The case $y=0$ is much easier.  We simply change variables $x \rightarrow x/a^2$, apply the definition \eqref{eq:Vdef}, reverse the orders of integration, and evaluate the $x$-integral.
\end{proof}
Let $M_N(k=0)$ denote the contribution to $M_N$ from $k=0$.  Applying Lemma \ref{lemma:H} to \eqref{eq:MNafterPoisson}, we have
\begin{equation}
 M_N(k=0) = \frac{1}{2} \sum_{\substack{(a,2l) = 1 \\ a \leq Y}} \frac{\mu(a)}{a^2}   
 \frac{1}{2 \pi i} \int_{(c_s)} \frac{G(s)}{s} g_{\alpha, \beta, \gamma}(s) 
\widetilde{F}(1+\frac{3s}{2})
\sum_{(n,2a)=1} \frac{\sigma_{\alpha,\beta,\gamma}(n)}{n^{\thalf+s}}  \frac{G_0(nl)}{nl}   ds.
\end{equation}
By Lemma \ref{lemma:Gk}, we have $G_0(nl) = \phi(nl)$ if $nl$ is a square, and is zero otherwise.  This means we can write $n \rightarrow l_1 n^2$, and then we obtain
\begin{equation}
\label{eq:MNk=0term}
 M_N(k=0) = \frac{1}{2} \sum_{\substack{(a,2l) = 1 \\ a \leq Y}} \frac{\mu(a)}{a^2} \frac{1}{2 \pi i} \int_{(\varepsilon)} \frac{G(s)}{s} g_{\alpha, \beta, \gamma}(s) \widetilde{F}(1 + \frac{3s}{2}) D_N(k=0,s) ds,
\end{equation}
where
\begin{equation}
 D_N(k=0;s) = \sum_{(n,2a) = 1} \frac{\sigma_{\alpha, \beta, \gamma}(l_1 n^2)}{(l_1 n^2)^{\frac12 + s}} \frac{\phi(ln)}{ln}.
\end{equation}

Now consider the terms with $k \neq 0$, which we denote $M_N(k \neq 0)$.  Suppose $\epsilon \in \{ \pm  \}$ is the sign of $k$.  Then we obtain $M_N(k \neq 0) = M_N^{+}(k \neq 0) + M_N^{-}(k \neq 0)$, where by applying Lemma \ref{lemma:H} to \eqref{eq:MNafterPoisson} and reversing the orders of summation and integration (supposing $c_u = c_s \geq 3$, say), we have
\begin{multline}
 M_N^{\epsilon}(k \neq 0) = \frac{1}{2} \sum_{\substack{(a,2l) = 1 \\ a \leq Y}} \frac{\mu(a)}{a^2}  \leg{1}{2 \pi i}^2 \int_{(c_s)} \int_{(c_u)} \widetilde{F}(1+u) \frac{G(s)}{s} g_{\alpha, \beta, \gamma}(s) 
\leg{l a^2}{ \pi}^{\frac{3s}{2} - u}
\\
\Gamma(\frac{3s}{2} - u) (\cos + \epsilon \sin)(\frac{\pi}{2}(\frac{3s}{2} - u))
\sum_{(n,2a)=1} \sum_{k \geq 1} 
\frac{\sigma_{\alpha,\beta,\gamma}(n)}{n^{\thalf+u-\frac{s}{2}} |k|^{\frac{3s}{2} -u}}    \frac{(-1)^k G_{\epsilon k}(nl)}{nl}
  ds du.
\end{multline}
Next we argue with inclusion-exclusion to treat the term $(-1)^k$.  
Suppose that $f(k)$ is some function such that $f(4k) = 4^{-z} f(k)$; in our case, we have $f(k) = G_{\epsilon k}(nl)/|k|^{\frac{3s}{2} -u}$, and $z= \frac{3s}{2} - u$.  Note that $nl$ is odd and hence $G_{4 \epsilon k}(nl) = G_{\epsilon k}(nl)$.  Then we write $k = k_1 k_2^2$ with $k_1$ squarefree and separate $k_1$ odd and even.  Then we have
\begin{align}
 \sum_{k=1}^{\infty} (-1)^k f(k) = \sumstar_{k_1 \text{ even}} \sum_{k_2} f(k_1 k_2^2) + \sumstar_{k_1 \text { odd}} \sum_{k_2=1}^{\infty} (-1)^{k_2} f(k_1 k_2^2)
\\
= 
\sumstar_{k_1 \text{ even}} \sum_{k_2=1}^{\infty} f(k_1 k_2^2) + \sumstar_{k_1 \text { odd}} (2\sum_{k_2=1}^{\infty}  f(4k_1 k_2^2) - \sum_{k_2=1}^{\infty} f(k_1 k_2^2) ).
\end{align}
Then using $f(4k) = 4^{-z} f(k)$, this becomes
\begin{equation}
\label{eq:-1ksum}
 \sum_{k=1}^{\infty} (-1)^k f(k) =  (2^{1-2z} - 1) \sumstar_{k_1 \text{ odd}}  \sum_{k_2=1}^{\infty} f(k_1 k_2^2) + \sumstar_{k_1 \text{ even}} \sum_{k_2=1}^{\infty} f(k_1 k_2^2)
\end{equation}
Applying the decomposition \eqref{eq:-1ksum} to $M_N^{\epsilon}(k \neq 0)$, we obtain
\begin{equation}
\label{eq:MNwithM1}
 M_N^{\epsilon}(k \neq 0) = \half \sum_{\substack{(a,2l) = 1 \\ a \leq Y}} \frac{\mu(a)}{a^2} (\sumstar_{k_1 \text{ odd}} \mathcal{M}_1(s,u, k_1, l) + \sumstar_{k_1 \text{ even}} \mathcal{M}_2(s,u, k_1, l) ),
\end{equation}
where
\begin{multline}
\label{eq:mathcalM1}
 \mathcal{M}_1(s,u, k_1, l) = \leg{1}{2 \pi i}^2 \int_{(c_s)} \int_{(c_u)} \widetilde{F}(1+u) \frac{G(s)}{s} g_{\alpha, \beta, \gamma}(s) 
\leg{l a^2}{ \pi k_1 }^{\frac{3s}{2} - u} (2^{1-3s+2u}-1)
\\
\Gamma(\frac{3s}{2} - u) (\cos + \epsilon \sin)(\frac{\pi}{2}(\frac{3s}{2} - u)) J_{\epsilon k_1}(3s-2u, \thalf + u - \tfrac{s}{2})
  ds du,
\end{multline}
and where
\begin{equation}
\label{eq:Jdef}
 J_{\epsilon k_1}(v, w) =  \sum_{(n,2a)=1} \sum_{k_2 \geq 1} 
\frac{\sigma_{\alpha,\beta,\gamma}(n)}{n^{w}  k_2^{v}}    \frac{G_{\epsilon k_1 k_2^2}(nl)}{nl}.
\end{equation}
We suppress the dependence on $l$, $a$, $\alpha$, $\beta$, and $\gamma$ in our notation for $J$.  The formula for $\mathcal{M}_2(s,u,k_1,l)$ is identical to \eqref{eq:mathcalM1} except the factor $(2^{1-3s+2u} - 1)$ is omitted.

\subsection{Continuation of $J$}
Here we develop the analytic properties of $J_{\epsilon k_1}(v, w)$.
\begin{mylemma}
\label{lemma:Jprop}
 Suppose that $\epsilon = \pm 1$, $(l, 2a) =1$, $k_1$ is squarefree, and $J_{\epsilon k_1}(v, w)$ is given initially by \eqref{eq:Jdef} for say $\text{Re}(v) > 2$ and $\text{Re}(w) > 2$.  Then for any $\delta > 0$, $J_{\epsilon k_1}(v, w)$ has a meromorphic continuation to $\text{Re}(w) > \delta$ and $\text{Re}(v) \geq 2$, provided $\alpha, \beta, \gamma$ are small enough compared to $\delta$.  Furthermore, in this region we have
\begin{equation}
 J_{\epsilon k_1}(v, w) = L_{2al}(\thalf + w + \alpha, \chi_{\epsilon k_1}) L_{2al}(\thalf + w + \beta, \chi_{\epsilon k_1})L_{2al}(\thalf + w + \gamma, \chi_{\epsilon k_1}) I_{\epsilon k_1}(v, w),
\end{equation}
where $\chi_{\epsilon k_1}(p) = \leg{\epsilon k_1}{p}$, the subscript on $L_{2al}$ means the Euler factors dividing $2 al $ are removed, and where $I_{\epsilon k_1}(v, w)$ is analytic in this domain and satisfies the bound
\begin{equation}
 I_{\epsilon k_1}(v, w) \ll_{\delta, \varepsilon} l^{-\half + \varepsilon} .
\end{equation}

\end{mylemma}
\begin{proof}
We begin by observing that $J$ has an Euler product as we now explain. Lemma \ref{lemma:Gk} immediately shows the multiplicativity in terms of $n$.  It follows from the definition of $G_k(n)$ (and a change of variables) that if $k$ is multiplied by a square coprime to $n$ then the value of $G_k(n)$ is unchanged.

According to the Euler product, we write $J_{\epsilon k_1}(v, w) = \prod_{p} J_{\epsilon k_1 ;p}(v, w)$.  In the easiest case $p | 2a$, we have
\begin{equation}
 J_{\epsilon k_1 ;p}(v, w) = \sum_{j=0}^{\infty} \frac{1}{p^{jv}} = (1-p^{-v})^{-1}.
\end{equation}
If $p \nmid 2a$, and if say $p^{l_p} || l$, then 
\begin{equation}
 J_{\epsilon k_1 ;p}(v, w) = \sum_{n=0}^{\infty} \sum_{k_2 = 0}^{\infty} \frac{\sigma_{\alpha,\beta,\gamma}(p^n)}{p^{nw + k_2 v}}    \frac{G_{\epsilon k_1 p^{2k_2}}(p^{n+l_p})}{p^{n+l_p}}.
\end{equation}
Suppose $p \nmid k_1$.  
In view of Lemma \ref{lemma:Gk}, there are two classes of terms where the Gauss sum does not vanish, one of which is for $n+l_p$ even and $n+l_p \leq 2k_2$, and the other is for $n +l_p= 2k_2 + 1$. Thus
\begin{equation}
 J_{\epsilon k_1 ;p}(v, w) = \sum_{n \equiv l_p \shortmod{2}} \sum_{2k_2 \geq n + l_p}  \frac{\sigma_{\alpha,\beta,\gamma}(p^{n})}{p^{nw + k_2 v}}    \frac{\phi(p^{n+l_p})}{p^{n+l_p}} +  \frac{(\frac{\epsilon k_1}{p})}{\sqrt{p}} \sum_{2k_2 \geq l_p -1}^{\infty} \frac{\sigma_{\alpha,\beta,\gamma}(p^{2k_2 + 1-l_p})}{p^{(2k_2 + 1-l_p)w + k_2 v}}.
\end{equation}
Executing the sum over $k_2$ in the first sum above, we obtain
\begin{equation}
\label{eq:JpEuler}
 J_{\epsilon k_1 ;p}(v, w) = (1-p^{-v})^{-1} \sum_{n\equiv l_p \shortmod{2}} \frac{\sigma_{\alpha,\beta,\gamma}(p^{n})}{p^{nw +  (\frac{n+l_p}{2}) v}}    \frac{\phi(p^{n+l_p})}{p^{n+l_p}} +  \frac{(\frac{\epsilon k_1}{p})}{\sqrt{p}} \sum_{2k_2 \geq l_p -1}^{\infty} \frac{\sigma_{\alpha,\beta,\gamma}(p^{2k_2 + 1-l_p})}{p^{(2k_2 + 1-l_p)w + k_2 v}}.
\end{equation}
If $p | k_1$ then the calculation is similar except we use the third line of \eqref{eq:Gk}.  The two classes of solutions come from $n+l_p$ even and $n+l_p \leq 2k_2 + 1$, and from $n+l_p = 2k_2 + 2$.  The former class gives an identical contribution to the previous case since for $n+l_p$ even the condition $n+l_p \leq 2k_2 + 1$ is equivalent to $n+l_p \leq 2k_2$.  Therefore in the case $p|k_1$ we have
\begin{equation}
 J_{\epsilon k_1 ;p}(v, w) = (1-p^{-v})^{-1} \sum_{n\equiv l_p \shortmod{2}} \frac{\sigma_{\alpha,\beta,\gamma}(p^{n})}{p^{nw +  (\frac{n+l_p}{2}) v}}    \frac{\phi(p^{n+l_p})}{p^{n+l_p}} -  p^{-1} \sum_{\substack{2k_2 +2 \geq l_p \\ k_2 \geq 0}} \frac{\sigma_{\alpha,\beta,\gamma}(p^{2k_2 + 2-l_p})}{p^{(2k_2 + 2-l_p)w + k_2 v}}.
\end{equation}

Now we analyze the above representations in closer detail.  Suppose $p \nmid 2a k_1 l$.  Then
\begin{equation}
 J_{\epsilon k_1 ;p}(v, w) = (1-p^{-v})^{-1} [1 + O(p^{-v-2w + \varepsilon}) + \frac{(\frac{\epsilon k_1}{p})}{ p^{\half + w}} (1-p^{-v}) (p^{-\alpha} + p^{-\beta} + p^{-\gamma} + O(p^{-v-2w + \varepsilon}))].
\end{equation}
Here $\varepsilon$ is small and accounts for the shift parameters.  If $\text{Re}(v) \geq 2$ and $\text{Re}(w) > \delta$, then
\begin{equation}
 J_{\epsilon k_1 ;p}(v, w) = [1 + (\frac{\epsilon k_1}{p}) p^{-\half - w} (p^{-\alpha} + p^{-\beta} + p^{-\gamma}) + O(p^{-2 + \varepsilon})],
\end{equation}
which we write as
\begin{equation}
\label{eq:Jpgeneric}
 (1-\frac{(\frac{\epsilon k_1}{p})}{p^{\half + w +\alpha}})^{-1}
(1-\frac{(\frac{\epsilon k_1}{p})}{p^{\half + w +\beta}})^{-1} 
(1-\frac{(\frac{\epsilon k_1}{p})}{p^{\half + w +\gamma}})^{-1}  (1 + O(p^{-1-2w + \varepsilon}) + O(p^{-2+\varepsilon})) ).
\end{equation}
The point is that the error term, when multiplied over all primes $p \nmid 2al k_1$ is $O_{\delta}(1)$.

Now consider the case $p \nmid 2al$, $p | k_1$.  A short calculation shows
\begin{equation}
 J_{\epsilon k_1 ;p}(v, w) = (1 + O(p^{-1-2w + \varepsilon}) + O(p^{-2 + \varepsilon})),
\end{equation}
which pleasantly agrees with \eqref{eq:Jpgeneric} in case $p | k_1$.

Now suppose that $p \nmid 2ak_1$, $p |l$, so $l_p \in \{ 1, 2 \}$.  If $l_p=1$ then
\begin{equation}
 J_{\epsilon k_1 ;p}(v, w) = O(p^{-w-v + \varepsilon}) + O(p^{-\half}) 
\end{equation}
Similarly, if $l_p = 2$ then
\begin{equation}
 J_{\epsilon k_1 ;p}(v, w) = O(p^{-v}) + O(p^{-\half-w-v + \varepsilon}).
\end{equation}
For each value of $l_p$, then, we have $J_{\epsilon k_1 ;p}(v, w) \ll p^{-l_p/2}$ in $\text{Re}(v) \geq 2$, $\text{Re}(w) > \delta$.

Finally, consider the case $p \nmid 2a$, $p|k_1$, $p|l$.  If $l_p =1$ then
\begin{equation}
 J_{\epsilon k_1 ;p}(v, w) = O(p^{-w-v + \varepsilon}) + O(p^{-1-w+\varepsilon}),
\end{equation}
while if $l_p=2$ then
\begin{equation}
 J_{\epsilon k_1 ;p}(v, w) = O(p^{-v}) + O(p^{-1}).
\end{equation}
For both values of $l_p$, we have $J_{\epsilon k_1 ;p}(v, w) \ll p^{-l_p/2}$ in $\text{Re}(v) \geq 2$, $\text{Re}(w) > \delta$.

Gathering all these results completes the proof.
\end{proof}

\subsection{Proof of Lemma \ref{lemma:MNresult}}
\label{section:proofofMNresult}
We now have the ingredients in place to prove Lemma \ref{lemma:MNresult}.
For notational simplicity, consider the contribution from $k_1$ odd in \eqref{eq:MNwithM1}, namely
\begin{equation}
 \half \sum_{\substack{(a,2l) = 1 \\ a \leq Y}} \frac{\mu(a)}{a^2} \sumstar_{k_1 \text{ odd}} \mathcal{M}_1(s,u, k_1, l).
\end{equation}
The contribution from $k_2$ even follows similar lines and we omit its discussion.
We write
\begin{multline}
\label{eq:M1evaluated}
 \mathcal{M}_1(s,u, k_1, l) = 
\leg{1}{2 \pi i}^2 \int_{(c_s)} \int_{(c_u)} H(s,u)
\leg{l a^2}{ k_1 }^{\frac{3s}{2} - u} I_{\epsilon k_1}(3s-2u, \thalf + u - \tfrac{s}{2})
\\
L_{2al}(1 + u - \frac{s}{2} + \alpha, \chi_{\epsilon k_1}) L_{2al}(1 + u - \frac{s}{2} + \beta, \chi_{\epsilon k_1})L_{2al}(1 + u - \frac{s}{2} + \gamma, \chi_{\epsilon k_1}) 
  ds du,
\end{multline}
where
\begin{equation}
 H(s,u) = \widetilde{F}(1+u) \frac{G(s)}{s} g_{\alpha, \beta, \gamma}(s) 
\pi^{-\frac{3s}{2} + u} (2^{1-3s+2u}-1)
\Gamma(\frac{3s}{2} - u) (\cos + \sin)(\frac{\pi}{2}(\frac{3s}{2} - u)).
\end{equation}
Note that $H$ has rapid decay in $s$ and $u$ in any vertical strips.  We initially have $c_s = c_u = 3$, say.  Next we move the contours (simultaneously) to $c_s = c_u = \half + \varepsilon$, in order to not cross any poles and to remain in a region where $J_{\epsilon k_1}$ is known to be analytic.  Next we move $c_u$ to $-\frac14 + \varepsilon$; in doing so we cross poles of the Dirichlet $L$-functions at $u=\frac{s}{2} - \alpha$, $u=\frac{s}{2} - \beta$, and $u=\frac{s}{2} - \gamma$ for $\epsilon = k_1 = 1$ only.  We denote the contribution to $M_N$ from these three residues as $M_N(k=\square, \alpha)$, $M_N(k=\square, \beta)$, and $M_N(k=\square, \gamma)$, respectively.  We shall develop these terms further in Section \ref{section:1.5miracle} and now focus on the error term coming from the new contour.  The quadratic large sieve inequality of Heath-Brown \cite{H-B} shows
\begin{equation}
 \sumstar_{K < k_1 \leq 2K} |L_{2al}(\sigma + it)|^4 \ll (al)^{\varepsilon} K^{1+\varepsilon} (1+|t|)^{1+\varepsilon},
\end{equation}
for $\half \leq \sigma \leq 1$.  Thus the sum over $k_1$ squarefree converges absolutely along these lines of integration.  Furthermore, notice that with these choices of $c_u$ and $c_s$ that
\begin{equation}
 H(s,u) \ll X^{3/4 + \varepsilon} (1 + |s|)^{-10^{10}} (1 + |u|)^{-10^{10}}.
\end{equation}
 Hence the contribution to $M_N$ from these error terms is
\begin{equation}
 \ll \sum_{a \leq Y} a^{-2} (l a^2)^{1 + \varepsilon} l^{-\half + \varepsilon} X^{3/4 + \varepsilon} \ll l^{1/2 + \varepsilon} Y X^{3/4 + \varepsilon}.
\end{equation}
This is precisely what we needed to show for Lemma \ref{lemma:MNresult}.  
It seems plausible that one could replace the application of Heath-Brown's quadratic large sieve with some of the work developed here since to a first approximation we obtain a third moment of quadratic Dirichlet $L$-functions in the above analysis.  Strictly speaking, this is problematic because of the ``bad'' Euler factors and the presence of the factor $I_{\epsilon k_1}$ which does not match our intial setup.

For future reference, we record the following expression for $M_N(k=\square, \alpha)$:
\begin{multline}
\label{eq:MNksquarealpha}
 M_N(k=\square, \alpha) = \half \sum_{\substack{(a,2l) = 1 \\ a \leq Y}} \frac{\mu(a)}{a^2} \frac{1}{2 \pi i} \int_{(c_s)} \frac{G(s)}{s} g_{\alpha, \beta, \gamma}(s) \widetilde{F}(1 + \tfrac{s}{2} - \alpha) 
\leg{l a^2}{ \pi }^{s + \alpha} (2^{1-2s - 2\alpha)}-1)
\\
\Gamma(s+\alpha) (\cos + \sin)(\frac{\pi}{2}(s + \alpha)) \text{Res}_{w=\half - \alpha} J_{1}(2s + 2\alpha, w)
  ds.
\end{multline}

\section{Proof of Lemma \ref{lemma:MNk0andMR111}}
\label{section:proofofMNk0andMR111}
\subsection{Part 1}
\label{section:firstmiracle}
Here we prove that \eqref{eq:MNk0andMR111} holds. 
We use the expression \eqref{eq:MNk=0term}, and compare it to
\begin{equation}
 M_R(1,1,1) = \frac{1}{2} \sum_{\substack{(a,2l) = 1 \\ a > Y}} \frac{\mu(a)}{a^2} \frac{1}{2 \pi i} \int_{(\varepsilon)} \frac{G(s)}{s} g_{\alpha, \beta, \gamma}(s) \widetilde{F}(1 + \frac{3s}{2}) D_R(1,1,1;s) ds,
\end{equation}
where
\begin{multline}
 D_R(1,1,1;s) = \frac{1}{\zeta_{2l}(2)} \sum_{(b,2l) = 1} \frac{1}{b^2} \sum_{r_1, r_2, r_3 | ab} \frac{\mu(r_1) \mu(r_2) \mu(r_3)}{r_1^{1/2+\alpha + s} r_2^{1/2 + \beta + s} r_3^{1/2 + \gamma + s}} 
\\
\sum_{(n,2) = 1} \frac{\sigma_{\alpha, \beta, \gamma}((lr_1 r_2 r_3)^* n^2)}{(l r_1 r_2 r_3)^* n^2)^{\frac12 + s}} \prod_{p | n l r_1 r_2 r_3} (1+p^{-1})^{-1}.
\end{multline}
Using a computer algebra package, we shall check that in fact
\begin{equation}
\label{eq:DNk0=DR111}
 D_N(k=0;s) = D_R(1,1,1;s).
\end{equation}
Notice that each Dirichlet series above has an Euler product and hence it suffices to check each Euler factor.  For $p \nmid a l$, we have that the Euler factor at $p$ for $D_N(k=0;s)$ is
\begin{equation}
\label{eq:EulerpDNk0}
 1 + (1-p^{-1}) \sum_{j \geq 1} \frac{1}{p^{j(1 + 2s)}} \sum_{a + b + c = 2j} p^{-a \alpha - b \beta - c \gamma}.
\end{equation}
Set $x = p^{-\half - \alpha -s}$, $y = p^{-\half - \beta -s}$, $z = p^{-\half - \gamma -s}$, and define
\begin{equation}
 T(x,y,z) = \sum_{a, b, c \geq 0} \frac{1 + (-1)^{a+b+c}}{2} x^a y^b z^c, \qquad U(x,y,z) = \sum_{a, b, c \geq 0} \frac{1 - (-1)^{a+b+c}}{2} x^a y^b z^c,
\end{equation}
which evaluate as rational functions in $x,y,z$.  
Then \eqref{eq:EulerpDNk0} takes the form
\begin{equation}
 1 + (1-p^{-1}) (-1 + T(x,y,z)).
\end{equation}
On the other hand, we compute the Euler factor at $p$ for $D_R(1,1,1;s)$ for $p \nmid a l$ as
\begin{equation}
\label{eq:EulerpDR111}
 1 + (1+p^{-1})^{-1} (-1 + Q(x,y,z)),
\end{equation}
where $Q$ corresponds to the same sum but with the annoying factor $(1+p^{-1})^{-1}$ removed, and a direct calculation gives 
\begin{equation}
 Q(x,y,z) = (1-p^{-2})[T + \frac{1}{p^2(1-p^{-2})} (T - (x+y+z)U + (xy + xz + yz)T - xyz U)].
\end{equation}
Now we obtain an explicit rational function representation for \eqref{eq:EulerpDR111} and a computer quickly verifies that these are identical.  Similar computations cover the cases $p | a$ and $p | l$ (the final case $p=2$ is trivial, both sides being $1$), giving 
\eqref{eq:DNk0=DR111}.

Thus we obtain that
\begin{equation}
 M_N(k=0) + M_R(1,1,1) = \frac{1}{2} \sum_{(a,2l) = 1} \frac{\mu(a)}{a^2} \frac{1}{2 \pi i} \int_{(\varepsilon)} \frac{G(s)}{s} g_{\alpha, \beta, \gamma}(s) \widetilde{F}(1 + \frac{3s}{2}) D_N(k=0;s) ds.
\end{equation}
Noting that
\begin{equation}
\frac{\phi(ln)}{ln} \sum_{(a, 2ln) = 1} \frac{\mu(a)}{a^2} = \frac{1}{\zeta_2(2)} \prod_{p | nl} (1+p^{-1})^{-1},
\end{equation}
and that 
\begin{equation}
  \sum_{(n,2) = 1} \frac{\sigma_{\alpha, \beta, \gamma}(l_1 n^2)}{(l_1 n^2)^{\frac12 + s}} \prod_{p | nl} (1+p^{-1})^{-1}  = \frac{1}{\sqrt{l_1}} A_{\alpha+s, \beta+s, \gamma+s}(l),
\end{equation}
we have
\begin{equation}
 M_N(k=0) + M_R(1,1,1) = \frac{1}{2 \zeta_2(2) \sqrt{l_1}}  \frac{1}{2 \pi i} \int_{(\varepsilon)} \frac{G(s)}{s} g_{\alpha, \beta, \gamma}(s) \widetilde{F}(1 + \frac{3s}{2}) A_{\alpha+s, \beta+s, \gamma+s}(l) ds .
\end{equation}
By Lemma \ref{lemma:A} and Remark \ref{remark:zero}, we can move the contour of integration to $-1/2 + \varepsilon$ crossing a pole at $s=0$ only in the process.  The residue at $s=0$ is as desired, and the error term is of size $O(X^{1/4+\varepsilon} l^{\varepsilon})$, also as desired.

\subsection{Part 2}
\label{section:1.5miracle}
Here we prove \eqref{eq:MNksquareandMR-111}; this is much more intricate than proving \eqref{eq:MNk0andMR111}.  We begin with an Archimedean-type identity.
\begin{mylemma}
\label{lemma:archcalc}
 Let $u$ be a complex number.  Then
\begin{equation}
\label{eq:archcalc}
 (2^{1-2u}-1) (\cos + \sin)(\frac{\pi}{2} u) \pi^{-u} \Gamma(u) = 2 \Gamma_u \frac{\zeta_2(1-2u)}{\zeta(2u)},
\end{equation}
where recall $\Gamma_u$ is defined by \eqref{eq:Gammadef}.
\end{mylemma}
\begin{proof}
 We first use the functional equation for the Riemann zeta function in the form
\begin{equation}
 \pi^{-u} \Gamma(u) \zeta(2u) = \pi^{-\half + u} \Gamma(\thalf - u) \zeta(1-2u),
\end{equation}
giving that the left hand side of \eqref{eq:archcalc} is
\begin{equation}
 (2^{1-2u}-1) (\cos + \sin)(\frac{\pi}{2} u) \pi^{-\half+u} \Gamma(\thalf - u) \frac{\zeta(1-2u)}{\zeta(2u)}.
\end{equation}
By examining the Euler product, we observe that
\begin{equation}
 (2^{1-2u} -1) \zeta(1-2u) = 2^{1-2u} \zeta_2(1-2u),
\end{equation}
so that the left hand side of \eqref{eq:archcalc} becomes
\begin{equation}
 2^{1-2u} (\cos + \sin)(\frac{\pi}{2} u) \pi^{-\half+u} \Gamma(\thalf - u) \frac{\zeta_2(1-2u)}{\zeta(2u)}.
\end{equation}
Next we use the trigonometric identity
\begin{equation}
 \cos(\theta) + \sin(\theta) = \sqrt{2} \cos(\frac{\pi}{4} - \theta),
\end{equation}
to get that the LHS of \eqref{eq:archcalc} is
\begin{equation}
 2^{\frac32-2u} \cos(\frac{\pi}{2} (\thalf -u)) \pi^{-\half+u} \Gamma(\thalf - u) \frac{\zeta_2(1-2u)}{\zeta(2u)}.
\end{equation}
Next we use the gamma function identity
\begin{equation}
 \pi^{-\half} 2^{1-v} \cos(\frac{\pi}{2} v) \Gamma(v) = \frac{\Gamma(\frac{v}{2})}{\Gamma(\frac{1-v}{2})},
\end{equation}
with $v= \half -u $, getting now that \eqref{eq:archcalc} is
\begin{equation}
 2^{1-3u}  \pi^{u} \frac{\Gamma(\frac{\half-u}{2})}{\Gamma(\frac{\half+u}{2})} \frac{\zeta_2(1-2u)}{\zeta(2u)}.
\end{equation}
Now from the definition of \eqref{eq:Gammadef}, we obtain the right hand side of \eqref{eq:archcalc}, as desired.
\end{proof}

We begin our proof of \eqref{eq:MNksquareandMR-111} by applying Lemma \ref{lemma:archcalc} to \eqref{eq:MNksquarealpha}, with $u=s + \alpha$, obtaining
\begin{multline}
\label{eq:MNksquarealpha2}
 M_N(k=\square, \alpha) = \half \sum_{\substack{(a,2l) = 1 \\ a \leq Y}} \frac{\mu(a)}{a^2} \frac{1}{2 \pi i} \int_{(c_s)} \frac{G(s)}{s} g_{\alpha, \beta, \gamma}(s) \widetilde{F}(1 + \tfrac{s}{2} - \alpha) \Gamma_{s + \alpha}
\\
2 (l a^2)^{s + \alpha}   \zeta_2(1-2\alpha - 2s) \text{Res}_{w=\half - \alpha} \frac{J_{1}(2s + 2\alpha, w)}{\zeta(2s + 2\alpha)}
  ds.
\end{multline}
One could wonder why it was beneficial to apply Lemma \ref{lemma:archcalc}.  The answer is that we wish to combine this term with $M_R(-1,1,1)$ which has a similar weight function to that appearing in \eqref{eq:MNksquarealpha2}.  This idea was used in \cite{Y}.

Recall that $c_s = \half + \varepsilon$.  It is convenient to represent the residue of $J_1(2s + 2\alpha, w)$ at $w = \half - \alpha$ by the value of $J_1(2s + 2 \alpha, w)/\zeta(\half + w + \alpha)$ at $w = \half - \alpha$.  Then we obtain
\begin{multline}
\label{eq:MNksquarealpha3}
 M_N(k=\square, \alpha) = \half \sum_{\substack{(a,2l) = 1 \\ a \leq Y}} \frac{\mu(a)}{a^2} \frac{1}{2 \pi i} \int_{(c_s)} \frac{G(s)}{s} g_{\alpha, \beta, \gamma}(s) \widetilde{F}(1 + \tfrac{s}{2} - \alpha) 
\\
\Gamma_{s + \alpha} a^{2 \alpha + 2s} D_N(k=\square, \alpha; s)
  ds.
\end{multline}
where
\begin{equation}
 D_N(k=\square, \alpha; s) = 2 l^{s + \alpha}   \zeta_2(1-2\alpha - 2s)  \frac{J_{1}(2s + 2\alpha, w)}{\zeta(2s + 2\alpha) \zeta(\half + w + \alpha)} \Big|_{w=\half - \alpha}.
\end{equation}
Before performing further analysis of this function, we recall from \eqref{eq:MRmainterms} that
\begin{multline}
M_R(-1, 1, 1) =  \frac{1}{2} \sum_{\substack{(a,2l) = 1 \\ a > Y}} \frac{\mu(a)}{a^2}
\frac{1}{2\pi i} \int_{(\varepsilon)} 
\frac{G(s)}{s} g_{\alpha,\beta,\gamma}(s) 
\widetilde{F}(1+ \tfrac{s}{2} - \alpha) 
\\
\Gamma_{\alpha + s} a^{2 \alpha + 2s}
D_R(-1,1,1;s)  
  ds.
\end{multline}
where
\begin{equation}
 D_R(-1,1,1;s) = \frac{1}{\zeta_2(2)} \sum_{\substack{(b,2l)=1}} \frac{1}{b^{2(1-\alpha-s)}} \sum_{r_1, r_2, r_3 | ab} \frac{\mu(r_1) \mu(r_2) \mu(r_3)}{r_1^{\half +\alpha + s}r_2^{\half +\beta + s}r_3^{\half +\gamma + s}} \frac{A_{-\alpha - s, \beta + s, \gamma + s}(lr_1 r_2 r_3)}{\sqrt{(lr_1 r_2 r_3)^*}}
 .
\end{equation}

By analogy with \eqref{eq:DNk0=DR111} and similar identities in \cite{Y}, it may not be surprising that the following miracle occurs:
\begin{mylemma}
\label{lemma:2Dirichletseries}
 We have
\begin{equation}
\label{eq:2Dirichletseries}
 D_{N}(k=\square, \alpha;s) = D_R(-1,1,1;s).
\end{equation}
\end{mylemma}
As in Section \ref{section:firstmiracle}, we shall verify this with a computer calculation in Section \ref{section:secondmiracle}.  Taking it for granted for now, we then have
\begin{multline}
M_N(k=\square, \alpha) + M_R(-1, 1, 1) =  \frac{1}{2} \sum_{\substack{(a,2l) = 1}} \frac{\mu(a)}{a^2}
\frac{1}{2\pi i} \int_{(\varepsilon)} 
\frac{G(s)}{s} g_{\alpha,\beta,\gamma}(s) 
\widetilde{F}(1+ \tfrac{s}{2} - \alpha) 
\\
\Gamma_{\alpha + s} a^{2 \alpha + 2s}
D_R(-1,1,1;s)
  ds.
\end{multline}
Grouping $ab$ into a variable, we see from the M\"{o}bius formula that only $ab=1$ survives, and hence $r_1 = r_2 = r_3 = 1$.  Thus
\begin{multline}
M_N(k=\square, \alpha) + M_R(-1, 1, 1) =  
\\
\frac{1}{2 \zeta_2(2) \sqrt{l_1}} 
\frac{1}{2\pi i} \int_{(\varepsilon)} 
\frac{G(s)}{s} g_{\alpha,\beta,\gamma}(s) 
\widetilde{F}(1+ \tfrac{s}{2} - \alpha) \Gamma_{\alpha + s} A_{-\alpha - s, \beta + s, \gamma + s}(l)
  ds.
\end{multline}
By Lemma \ref{lemma:A} and Remark \ref{remark:zero}, we can move the contour of integration to $-1/2 + \varepsilon$, crossing a pole at $s=0$ only.  The residue at $s=0$ gives the main term in \eqref{eq:MNksquareandMR-111}, and  the error term is of size $O(X^{3/4 + \varepsilon} l^{\varepsilon})$, as desired.

\subsection{Proof of Lemma \ref{lemma:2Dirichletseries}}
\label{section:secondmiracle}
The basic plan is similar to that used in Section \ref{section:firstmiracle}: check the identity at each Euler factor using a computer.  There is a minor issue with convergence that we delay explaining for a moment.  First consider the case $p=2$ (which can be easily done by hand).  The Euler factor for $D_R$ at $p=2$ is $1$, while the Euler factor for $D_N$ is
\begin{equation}
 2 (1-2^{-2s-2\alpha})(1-2^{-1})(1-2^{-2s-2\alpha})^{-1} = 1.
\end{equation}
Now we add the following definitions to those used in Section \ref{section:firstmiracle}: $w = p^{2 \alpha + 2s}$, 
\begin{equation}
 V(x,y,z) = \sum_{n=0}^{\infty} \frac{\sigma_{-\alpha-s, \beta+s, \gamma+s}(p^{2n})}{p^n} = \sum_{a, b, c \geq 0} \frac{1 + (-1)^{a+b+c}}{2} x^a w^a y^b z^c,
\end{equation}
and
\begin{equation}
 W(x,y,z) = \sum_{n=0}^{\infty} \frac{\sigma_{-\alpha-s, \beta+s, \gamma+s}(p^{2n+1})}{p^{n+\half}} = \sum_{a, b, c \geq 0} \frac{1 - (-1)^{a+b+c}}{2} x^a w^a y^b z^c.
\end{equation}
We view $x,y,z, p$ as free variables; note $w = 1/(px^2)$.  Then the Euler factor at $p|a$ for $D_N$ is $(1-p^{-1})(1-\frac{w}{p})^{-1}$, while for $D_R$ it is
\begin{equation}
 (1-p^{-2})(1-wp^{-2})^{-1} [ 1 + (1+p^{-1})^{-1} \{-1 + Q' \}],
\end{equation}
where
\begin{equation}
 Q' = V - (x+y+z)W + (xy+xz + yz)V - xyz W.
\end{equation}
A computer indeed verifies these are equal.  For $p \nmid 2al$, we have that the Euler factor at $p$ for $D_N$ is
\begin{equation}
 (1-\frac{w}{p})^{-1} (1-\frac{1}{w})(1-\frac{1}{p}) (A + B),
\end{equation}
where
\begin{equation}
 A = (1-\frac{1}{w})^{-1}(1 + (1-p^{-1})[-1 + T]), \quad \text{and} \quad B = w x U.
\end{equation}
The corresponding Euler factor at $p$ for $D_R$ is
\begin{multline}
 (1-p^{-2})\{A_{-\alpha-s, \beta + s, \gamma+s}(1) + \frac{w}{p^2}(1-\frac{w}{p^2})^{-1} [A_{-\alpha-s, \beta + s, \gamma+s}(1) 
\\
- \frac{(x+y+z + xyz)}{\sqrt{p}} A_{-\alpha-s, \beta + s, \gamma+s}(p) +(xy+xz+yz)A_{-\alpha-s, \beta + s, \gamma+s}(p^2)] \}.
\end{multline}
Then observe
\begin{gather}
 A_{-\alpha-s, \beta + s, \gamma+s}(1) = 1 + (1+p^{-1})^{-1}(-1 + V), 
\\
\frac{1}{\sqrt{p}} A_{-\alpha-s, \beta + s, \gamma+s}(p) = (1+p^{-1})^{-1} W, \qquad A_{-\alpha-s, \beta + s, \gamma+s}(p^2) = (1+p^{-1})^{-1} V.
\end{gather}
Again, a computer verifies the equality of Euler factors.  We omit the final case $p|l$ (two cases, actually, dependong on if $l_p = 1$ or $l_p = 2$) which follow similar lines.

Finally, we argue why it suffices to check that the Euler products agree at the desired point.  Towards this end, we have
\begin{mylemma}
\label{lemma:J1merocontinuation}
 There exists $\delta > 0$ such that for $\alpha, \beta, \gamma$ small enough compared to $\delta$ we have that
\begin{equation}
\label{eq:J1ratio}
 \frac{J_1(v,w)}{\zeta(v) \zeta(\thalf + w + \alpha) \zeta(\thalf + w + \beta) \zeta(\thalf + w + \gamma)  }
\end{equation}
has meromorphic continuation to the region $\text{Re}(v) > -\delta$, $\text{Re}(w) > \half - \delta$.  More precisely, \eqref{eq:J1ratio} equals, with $(\alpha_1, \alpha_2, \alpha_3)$ denoting $(\alpha, \beta, \gamma)$, 
\begin{equation}
\label{eq:J1ratiomainterm}
\frac{\prod_{1\leq i \leq j \leq 3} \zeta(v + 2w + \alpha_i + \alpha_j)}{\zeta(\thalf + w + v + \alpha)\zeta(\thalf + w + v + \beta)\zeta(\thalf + w + v + \gamma)}
\end{equation}
times an Euler product that is absolutely convergent in the above-stated domain.
\end{mylemma}
\begin{proof}
 It suffices to consider the behavior of the Euler factors with $p \nmid 2al$ in which case recall the Euler factor is given by \eqref{eq:JpEuler}.  This takes the shape
\begin{equation}
 (1-p^{-v})^{-1}[1 + (1-p^{-1}) \frac{\sigma_{\alpha, \beta, \gamma}(p^2)}{p^{v+2w}} + O(p^{-2v-4w + \varepsilon}) + (1-p^{-v}) p^{-\half - w} (\sigma_{\alpha, \beta, \gamma}(p) + O(p^{-v-2w + \varepsilon}))]
\end{equation}
From this description we can read off \eqref{eq:J1ratiomainterm}.
\end{proof}
By Lemmas \ref{lemma:J1merocontinuation} and \ref{lemma:A}, both sides of \eqref{eq:2Dirichletseries} can be continued meromorphically to a domain $\text{Re}(s) > -\delta$ with some $\delta > 0$.  Thus it suffices to check for each prime $p$ that the Euler factor at $p$ for each side agrees with the other side, as claimed.

\end{document}